\documentclass{article}
\usepackage{amsmath}      
\usepackage{amssymb}      
\usepackage{amsthm}       
\usepackage{mathtools}    
\usepackage{bm}           
\usepackage{graphicx}     
\usepackage{hyperref}     
\usepackage{geometry}
\usepackage{xcolor}
\usepackage{subcaption}
\usepackage{float}
\usepackage{comment}
\usepackage{booktabs}
\usepackage{tikz}
\usetikzlibrary{positioning}
\newtheorem{theo}{Theorem}
\newtheorem{prop}{Proposition}
\newtheorem{defn}{Definition}
\newtheorem{corollary}{Corollary}

\newtheorem{remark}{Remark}
\newtheorem{lemma}{Lemma}

\newcommand{\R}{\mathbb{R}}
\newcommand{\1}{\mathbf{1}}
\newcommand{\E}{\mathbb{E}}
\newcommand{\En}{\mathcal{E}}

\renewcommand{\L}{\mathcal{L}}

\newcommand{\diag}{\text{diag}}

\newcommand{\argmax}{\text{argmax}}

\usepackage{biblatex}
\title{Computing Stationary Distribution via Dirichlet-Energy Minimization by Coordinate Descent}

\date{July 2025}

\usepackage{authblk}

\author[1]{Konstantin Avrachenkov}
\author[2]{Lorenzo Gregoris}
\author[2]{Nelly Litvak}

\affil[1]{Centre Inria d’Universit\'e C\^ote d'Azur, NEO team, Sophia Antipolis, France. 
\texttt{k.avrachenkov@inria.fr}}
\affil[2]{Department of Mathematics and Computer Science, Eindhoven University of Technology, Eindhoven, The Netherlands. 
\texttt{\{l.gregoris,n.v.litvak\}@tue.nl}}

\begin{document}

\maketitle

\begin{abstract}
    We present an optimization-based formulation of the Red Light Green Light (RLGL) \cite{RLGL} algorithm for computing stationary distributions of large Markov chains. This perspective clarifies the algorithm’s behavior, establishes exponential convergence for a class of chains, and suggests practical scheduling strategies to accelerate convergence. 
\end{abstract}

\tableofcontents

\newpage

\section{Introduction}

Computing the stationary distribution of a Markov chain is a fundamental task with applications 
ranging from queueing systems, performance evaluation, and chemical networks, to large-scale problems such as PageRank \cite{brin1998anatomy}, semi-supervised learning \cite{avrachenkov2012generalized}, and
graph neural networks \cite{gasteiger2018predict, bojchevski2020scaling}.

Formally, the task is to solve the eigenvector problem  $\pi P=\pi$ for a transition probability matrix $P$. In many of these real world settings, a Markov chain may have billions of possible states, making direct numerical methods infeasible and leaving iterative algorithms as the only practical option.

Iterative algorithms start with a guess solution $\hat\pi_0$, and generate recursively a sequence $\{\hat\pi_t\}_{t\geq0}$, at each step updating the coordinates of $\hat\pi_t$, with the goal of eventually converging to the solution $\pi$. 
A large class of such iterative algorithms update their estimate using the residual
\begin{equation}
    r_t = \hat\pi_t(P - I).
    \label{eq:residual}
\end{equation}
Usually, the stopping condition involves some norm of the residual, i.e. we stop when $\Vert r_t\Vert \leq \varepsilon$ for some $\varepsilon > 0$.

Examples include OPIC \cite{OPIC}, Gauss--Southwell-type PageRank algorithms \cite{mcsherry_uniform_2005}, and more generally a variety of residual-elimination methods \cite{jeh2003scaling}. 

The RLGL (“Red Light Green Light”) algorithm \cite{RLGL} is a recent unifying framework for many iterative approaches. In RLGL, only a subset of coordinates is updated at each iteration according to a chosen schedule (following the terminology in \cite{RLGL}, the updated coordinates receive a “green light", thus the name of the algorithm). Different schedules and initializations correspond to many of the classical algorithms cited above. RLGL is remarkably simple and performs extremely well in practice, but proving convergence guarantees for the best performing schedules has been challenging. Nonetheless, in the experiments, RLGL consistently outperforms many known algorithms, including state of the art Krylov subspace methods such as GMRES \cite{GMRES}.
 
Another classical approach for stationary distribution computation is reformulating the eigenvector problem as a minimization problem of the least-squares objective 
\begin{equation}\label{eq:least_squares}
    f(x) = \frac12 \|x(P -I)\|^2,
\end{equation}
for some norm $\|\cdot\|$. Then $f$ is a convex function minimized at the solutions of the linear system, and we can apply the well-established theory for this class of optimization problems.
The simplest class of iterative optimization algorithms are first-order methods \cite{wright2015coordinatedescentalgorithms}, which use gradient information to update the current estimate. 

There are drawbacks when applying first-order methods to minimize  the least-squares objective function \eqref{eq:least_squares}. For instance, the gradient reads
\begin{equation}
    \nabla f(x) = x(P-I)(P-I)^\top,
\end{equation}
which depends not only on $P$, but also depends on $P P^\top$ and $P^\top$. This could be a problem because even if $P$ is sparse, $PP^\top $ may not be (an effect known as \emph{fill-in} \cite[Section~3.3.2]{saad2003iterative}. A way to address this is to avoid forming $PP^\top$ explicitly and instead compute $(xP)P^\top$ through two sparse matrix–vector multiplications. While this only doubles the cost per iteration, the condition number of $PP^\top$ is the square of that of $P$, which can slow down convergence. Finally, in applications like PageRank where accessing out-neighbors is cheaper than in-neighbors, approaches whose gradients (or updates) depend only on $P$ and not on $P^\top$, are operationally attractive.

Another formulation as an optimization problem applies to the special case when the matrix $A$ of a linear system $Ax=b$ is symmetric positive semi-definite. Under this assumption, one can construct a convex objective function, of which the gradient depends linearly on $A$:
\begin{equation}
    \Phi(x) = \frac12 x^\top A x - b^\top x.
\label{eq:Phi}
\end{equation}
This is usually referred to as the \emph{energy function} of the linear system. This formulation is the starting point of the conjugate-gradient method and other techniques for solving positive semi-definite linear systems \cite{CG}. The power of this approach is in the fact that the gradient of the energy function is
\begin{equation}\label{eq:linear_updates}
    \nabla \Phi(x_t) = Ax_t - b = -r_t.
\end{equation}
Relation \eqref{eq:linear_updates} bridges the optimization algorithms based on the gradient, and the iterative algorithms based on the residual. 

The limitation of this approach for stationary distribution computations is that, in general, transition matrices are not symmetric (in this trivial case the stationary distribution is uniform). Therefore, the conjugate-gradient framework is not directly applicable. The RLGL algorithm however, uses linear updates in the transition matrix $P$, suggesting there may exist an objective function of the form \eqref{eq:Phi}, for which a coordinate descent method produces similar updates and converges exponentially fast.

\paragraph{Contributions.}
In this paper, we provide partial theoretical explanation for the strong empirical performance of RLGL. Our main results are:

\begin{enumerate}
    \item \textbf{Variational formulation.}  
    We show that when $P$ is similar to a symmetric matrix---for example, for reversible Markov chains---the stationary distribution problem admits an energy formulation based on the \emph{Dirichlet energy}. In this setting, RLGL corresponds to a block coordinate descent method that minimizes this energy.

    \item \textbf{Exponential convergence for nearly reversible chains.}  
    Using this energy viewpoint, we prove that RLGL converges \emph{exponentially fast} for a class of chains we call \emph{nearly reversible}, with minimal assumption on the update schedule. This extends the previously known special cases in \cite{RLGL}.

    \item \textbf{Energy inspired heuristics.}  
    The variational formulation yields new coordinate selection rules that provably maximizes the decrease in energy at each step, which we call \emph{Gauss-Southwell-Dirichlet} heuristics. We show that these rules outperform the previously known heuristics.

\end{enumerate}

\paragraph{Outline of the paper.}  
In Section~\ref{sec:preliminaries} we start by introducing the notation, and give brief preliminaries on Markov chains, the RLGL algorithm, the Dirichlet energy of a Markov Chain, and coordinate descent methods to minimize convex functions. In Section~\ref{sec:cd_rev} we prove the equivalence of RLGL to coordinate descent in the case when the transition matrix is similar to a symmetric matrix. In Section~\ref{sec:nearly_rev} we extend these ideas by viewing a general irreversible Markov chain as a perturbed reversible chain. By doing so, we obtain spectral conditions that, when satisfied, guarantee exponential convergence. Finally, in Section~\ref{sec:numerical}, we propose new energy-inspired heuristics, and we compare them with baseline heuristics for computation of the stationary distribution and the PageRank computation on both real and synthetic networks.

\section{Preliminaries}
\label{sec:preliminaries}

This section recollects the main tools used in the paper, linking the probabilistic and optimization viewpoints. 
We first fix notation and recall basic facts on Markov chains, stationary distribution and reversibility. 
We then introduce the RLGL algorithm, a block-coordinate analogue of power iteration. 
Finally, we discuss the graph Laplacian and the Dirichlet energy, which provide the variational framework to interpret RLGL as a coordinate descent method.

\subsection{Notation and conventions}
We adopt the row vector convention for Markov chains: probability distributions are represented as row vectors $\pi \in \mathbb{R}^n$ satisfying $\pi \mathbf{1}^\top = 1$, where $\mathbf{1} \in \mathbb{R}^n$ denotes the all-ones row vector. A matrix $P \in \mathbb{R}^{n \times n}$ is called \emph{row stochastic} if $P \mathbf{1}^\top = \mathbf{1}^\top$ and $P_{ij} \geq 0$ for all $i,j$. 

For a positive integer $n$, we write $[n] := \{1,2,\dots,n\}$.  
We denote by $\{e_i\}_{i=1}^n$ the standard basis of $\mathbb{R}^n$, where $e_i$ is the row vector with a $1$ in position $i$ and $0$ elsewhere.  
Given a subset $B \subseteq [n]$, we let $I_B \in \mathbb{R}^{n \times n}$ be the diagonal \emph{projector} onto the coordinates in $B$, defined by
$$
  (I_B)_{ij} = \delta_{ij}\,{\mathbb I}(i \in B),
$$
where $\delta_{ij}$ is the Kronecker delta and ${\mathbb I}(\cdot)$ denotes the indicator function.  
For a matrix $M \in \mathbb{R}^{n \times n}$ and a subset $B \subseteq [n]$, we denote by $M_{BB}$ the \emph{principal submatrix} of $M$ with rows and columns indexed by $B$. Equivalently, $M_{BB} = (M_{ij})_{i,j \in B} \in \mathbb{R}^{|B|\times|B|}$.

Time or iteration indices are denoted by $t \in \mathbb{N} = \{0,1,2,\dots\}$, starting at $t=0$.  
For a differentiable function $f: \R^n\to \R$, we denote by $\nabla f \in \R^n$ its gradient, and by $\partial_i f := (\nabla f)_i$.

\subsection{Markov chains and stationary distribution}
A discrete-time Markov chain on the state space $V = [n]$ is specified by an $n \times n$ row-stochastic matrix $P$, where $P_{ij}$ denotes the probability of transitioning from state $i$ to state $j$. We always assume that the chain is irreducible, so that by the Perron--Frobenius theorem there exists a unique stationary distribution vector $\pi$, which is the left eigenvector of $P$ corresponding to eigenvalue $1$, i.e.,
$$
\pi = \pi P, 
\qquad 
\sum_{i=1}^n \pi_i = 1,
\qquad 
\pi_i > 0 \text{ for all } i \in V.
$$

A Markov chain is called \emph{reversible} with respect to $\pi$ if it satisfies the \emph{detailed balance equations}
\begin{equation}
    \pi_i P_{ij} = \pi_j P_{ji}
    \qquad \text{for all } i,j \in V.
\label{eq:DB}
\end{equation}

Given a Markov chain with stationary distribution $\pi$, we define its \emph{time-reversal} as the chain with transition matrix
$$
P^* = \Pi^{-1} P^\top \Pi,
$$
where $\Pi = \mathrm{diag}(\pi)$. Both $P$ and $P^*$ share the same spectrum and stationary distribution, and the chain is reversible if and only if $P = P^*$. Equivalently, $P$ is reversible if and only if it is self-adjoint in the Hilbert space $L^2(\pi)$, equipped with the rescaled inner product $\langle \cdot, \cdot \rangle_\pi$, defined for $f,g : [n]\to\R$ as
$$
\langle f, g \rangle_\pi := \sum_{x=1}^n \pi_xf(x)g(x).
$$

For our purposes, it is more convenient to work with the standard Euclidean inner product. This can be achieved by applying the isometry $x\mapsto x\Pi^{-1/2}$. In these new coordinates, which we denote by $y$, a reversible transition matrix $P$ with stationary distribution $\pi$ becomes symmetric.

It is often convenient to separate a transition matrix into its symmetric and antisymmetric components with respect to $\pi$. Indeed, any irreducible $P$ admits the decomposition
\begin{equation}
    P = \tfrac{1}{2}(P + P^*) \;+\; \tfrac{1}{2}(P - P^*),
\label{eq:rev_dec}
\end{equation}
where the first term is the \emph{reversible part} of $P$, itself a stochastic matrix, and the second term is a matrix with zero row sums that captures the deviation from reversibility.

\subsection{RLGL algorithm}

The Red-Light-Green-Light (RLGL) algorithm \cite{RLGL} works as follows.  Given an iterate $x_t$ and residue vector $r_t = x_t(P-I)$, the update is
\begin{equation}
    x_{t+1} = x_t - x_t(I-P)I_{B_t} = x_t + r_tI_{B_t}.
\label{eq:rlgl_update}
\end{equation}
Instead of recomputing the residue at every iteration, one can update it recursively as
\begin{align*}
r_{t+1} &= x_{t+1}(P-I) \\
&= (x_t - x_t(I-P)I_{B_t})(P-I) \\
&= r_t + r_tI_{B_t}(P-I)=r_tP(B_t),
\end{align*}
where $P(B_t)=I+I_{B_t}(P-I)$ is the matrix such that its rows $i \in B_t$ are the same as in $P$ and the rows $i\notin B_t$ are as in the identity matrix. Thus, RLGL 
can be viewed as block-coordinate variant of the power iterations. In other words, instead of updating all coordinates simultaneously as in Power iteration, RLGL in \eqref{eq:rlgl_update} updates only a subset $B_t \subseteq [n]$; when $B_t =[n]$ for all $t$, the method reduces to the classical power iteration. 

The total residual is conserved and equals to zero:
$$
r_{t}\mathbf{1}^\top=x_{t}(P-I)\mathbf{1}^\top
= x_{t}(P\mathbf{1}^\top - \mathbf{1}^\top)
= 0.
$$
The terminology in \cite{OPIC}, \cite{RLGL} emphasizes this by alluding to the residual as \emph{cash}, and interpreting the updates as cash movements or transactions. 

At each step, the residual may decrease due to cancellations, when cash of opposite sign meets at a node. Consequently, the $\ell_1$ norm of the residual is non-increasing:
$$
\|r_{t+1}\|_1 \le \|r_t\|_1, 
$$

The approximation of the stationary distribution is obtained by normalizing vector $x_t$, which in turn is expressed as the sum of updates at each step, as seen from \eqref{eq:rlgl_update}:
\[
x_t = x_0 + \sum_{k=1}^{t-1} r_k I_{B_k}, \quad \hat\pi_t=\frac{1}{\|x_{t}\|_1}\, x_t.  
\]
In other words, $x_t$ records the cumulative \emph{transaction history}, that is the total amount of  cash that each node has distributed up to time $t$. When the residue $r_t$ goes to zero, $\hat \pi_t$ converges to the stationary distribution. 

Thus, the crux of RLGL updates is to drain the cash using as small as possible number of transactions.  However, identifying an optimal sequence of transactions is a difficult optimization problem, intractable even for relatively small graphs. 
Consequently, practical implementations rely on heuristics, both randomized and deterministic.
Several of these heuristics are summarized in Table~\ref{tab:heuristics}.
\begin{table}[H]
\centering
\begin{tabular}{@{}ll@{}}
\toprule
\textbf{Heuristic} & \textbf{Selection Rule} \\ \midrule
\texttt{RR}    & \(B_t = \{ i \mid i \equiv t \, (\mathrm{mod}\, n) \}\) \\[1ex]
\texttt{Rand}  & \(B_t = \{ i \}\) with probability \(\frac{1}{n}\) \\[1ex]
\texttt{Greedy} & \(B_t = \arg\min_{i} \|r_t + r_t I_{\{i\}}(P-I)\|_1\) \\[1ex]
\texttt{MaxC}  & \(B_t = \{ i \mid i = \text{argmax}_{j} |r_{t,j}| \}\) \\[1ex]
\texttt{PC}    & \(B_t = \{ i \}\) with probability proportional to \(|r_{t,i}|\) \\[1ex]
\texttt{Theta} & 
\(
B_t = \left\{ i \mid i \equiv t \, (\mathrm{mod}\, n) \text{ and } |r_{t,i}| \ge \theta_t(r) \right\}, \quad \theta_t(r) = \left(\frac{1}{n}\sum_j |r_{t',j}|^r\right)^{1/r}
\) \\ \bottomrule
\end{tabular}
\caption{Summary of RLGL heuristics considered in \cite{RLGL}.}
\label{tab:baseline_heuristics}
\end{table}

Numerical evidence suggests that residue-dependent heuristics proposed in \cite{RLGL}, such as \texttt{Theta}, achieve state-of-the-art performance, often outperforming classical methods in practice.

\subsection{Dirichlet energy and the graph Laplacian}

The normalized Laplacian associated with an irreducible transition matrix $P$ is defined as
$$
\mathcal{L_{{\rm norm}}} \;=\; I - \Pi^{1/2}P\Pi^{-1/2}, \qquad \Pi = \mathrm{diag}(\pi),
$$
where $\pi$ is the stationary distribution of $P$. If $\lambda_1 = 1, \lambda_2, \dots, \lambda_n$ are the eigenvalues of $P$ then $1 - \lambda_i$ are the eigenvalues of $\mathcal{L_{{\rm norm}}}$. In particular, the kernel of $\mathcal{L_{{\rm norm}}}$ is one-dimensional, spanned by $\sqrt{\pi}$. 
Since $\mathcal{L_{{\rm norm}}}$ is in general non-symmetric, its eigenvalues may be complex numbers. A standard way to handle this is to consider the \emph{additive symmetrization}, or \emph{symmetrized Laplacian}:
$$
\mathcal{L}_{\mathrm{sym}} \;=\; \frac{1}{2} (\mathcal{L_{{\rm norm}}} + \mathcal{L_{{\rm norm}}}^\top)
\;=\; I - \Pi^{1/2} \Big(\frac{P+P^*}{2}\Big) \Pi^{-1/2}.
$$
The symmetrized Laplacian is symmetric and positive semi-definite, so its eigenvalues, denoted by $\mu_1,\dots,\mu_n$, are real and contained in $[0,2]$. Its kernel coincides with that of $\mathcal{L_{{\rm norm}}}$, spanned by $\sqrt{\pi}$.

The \textit{Dirichlet energy} associated with $P$ is the quadratic form induced by the symmetrized Laplacian:
\begin{equation}
    \En(y) = \frac12 y\mathcal{L}_{\mathrm{sym}}y^\top.
\label{eq:dir_energy}
\end{equation}
This function is convex, with the only flat direction being the one spanned by the kernel of the Laplacian. More precisely, restricted to the orthogonal complement of $\sqrt{\pi}$, the Dirichlet energy is $\mu_2$-strongly convex, where $\mu_2$ is the smallest non-zero eigenvalue of $\mathcal{L}_{\mathrm{sym}}$. The constant $\mu_2$ is often called the \textit{Poincaré constant} of $P$ \cite{salez2025modern}. In the reversible and lazy case, one has $\mu_2 = 1 - \lambda_2(P)$, so the Poincaré constant coincides with the usual spectral gap. 

For brevity, we will henceforth write $\L$ for the symmetrized Laplacian and denote by $\mu$ its smallest nonzero eigenvalue. 

The Dirichlet energy plays a central role in the analysis of mixing times: it controls the decay of variance and underlies functional inequalities (such as Poincaré and log-Sobolev inequalities) that quantify the speed of convergence to equilibrium \cite{levin2017markov}. The definition of the  Dirichlet energy in this paper differs slightly from the classical one, as we work in Euclidean coordinates rather than in the weighted space $\ell^2(\pi)$, and use the transpose of the symmetrized Laplacian; however, the essential properties remain unchanged.

The next proposition establishes an important relation between the squared $\ell_2$  norm of the the residual and the Dirichlet energy.
\begin{prop}
\label{prop:residual-energy}
    There exist constants $0 < c_1 \leq c_2 < +\infty$ such that
    $$
    c_1 \En(x) \leq \Vert x\Pi^{1/2}(P-I)\Vert_2^2 \leq c_2 \En(x).
    $$    
\end{prop}
\begin{proof}
    Apply Lemma~\ref{lemma:quad_bound} in the Appendix to the Dirichlet energy and $x\Pi^{1/2}(P-I)(P-I)^\top\Pi^{1/2}x^\top$. 
\end{proof}

In particular, Proposition~\ref{prop:residual-energy} implies that if either the energy or the residual goes to zero, then both these quantities go to zero together with the same rate. 

\subsection{Optimization with coordinate descent}
We begin by recalling standard results on coordinate descent methods for minimizing a differentiable function $f: \mathbb{R}^n \to \mathbb{R}$. A coordinate descent method generates a sequence of iterates via updates of the form:
\begin{equation}
    x_{t+1} = x_t - \alpha_t \partial_{i_t}f(x_t) e_{i_t},
\label{eq:coord_desc}
\end{equation}
where $i_t$ is the chosen coordinate at step $t$ and $\alpha_t\in \R$ is the step-size. 
It is known that coordinate descent methods converge exponentially, if $f$ is a so-called $\mu$-PL function (see Definition~\ref{def:pl_ineq} below) and the gradient of $f$ is Lipschitz continuous.

\begin{defn}[$\mu$-PL function, \cite{wright2015coordinatedescentalgorithms}]
Let $f$ be a differentiable function, and let $f^* = \inf_{x\in \mathbb{R}^n} f(x)$. We say $f$ is $\mu$-PL if it satisfies
$$
f(x) - f^* \leq \frac{1}{2\mu} \Vert \nabla f(x)\Vert_2^2, \qquad \forall x \in \R^n.
$$
\label{def:pl_ineq}
\end{defn}

The Polyak-Lojasiewicz inequality lower bounds the $\ell_ 2$ norm of the gradient with the function's suboptimality. 
In particular, any $\mu$-strongly convex function \eqref{eq:convexity} is $\mu$-PL; this is formally stated in  Lemma~\ref{lem:stronly-convex} in the Appendix. 

Although the Dirichlet energy \eqref{eq:dir_energy} is not strongly convex, due to the flat direction corresponding to stationary distribution, it still satisfies PL inequality with $\mu$ equal to the Poincaré constant, see Lemma~\ref{lem:PSD-PL} in the Appendix.

Let $L_{max} := \max_i L_i$, where $L_i$'s are the Lipschitz constants: 
\begin{equation}
    \vert \partial_i f(x + h e_i) - \partial_i f(x) \vert \leq L_i |h|, \qquad i=1,\dots,n.
\label{eq:lip}
\end{equation}
Functions that satisfy \eqref{eq:lip} have quadratically bounded growth as stated in Lemma~\ref{lemma:descent} below.

\begin{lemma}
Let $f: \mathbb{R}^n\to \mathbb{R}$ be differentiable with coordinate-wise Lipschitz constants $L_1,\dots,L_n$. Then,
$$
f(x+h e_i) \leq f(x) + h\partial_i f(x) + \frac{L_i}{2}h^2, \qquad \forall h \in \R.
$$
\label{lemma:descent}
\end{lemma}
\begin{proof}
See Appendix \ref{sec:appendix}.
\end{proof}

 Using Lemma \ref{lemma:descent} with $h=-\alpha_t \partial_{i_t}f(x_t)$, one can show that the step size that maximizes the guaranteed one-step progress, i.e. the decrease of the objective function in one step, is $\alpha_t =  1/L_{i_t}$. This step-size yields a guarantee one-step decrease of:
\begin{equation}
    f(x_t) - f(x_{t+1}) \geq \frac{1}{2L_{i_t}} |\partial_{i_t}f(x_t)|^2.
\label{eq:dec_guar}
\end{equation}

From Eq. \eqref{eq:dec_guar}, we see that updating a coordinate with big gradient component seems like a good idea. If this component has a big enough fraction of the full gradient norm, then we can use PL-inequality (Definition~\ref{def:pl_ineq}) to relate the one-step progress to the current sub-optimality $f(x_t)-f^*$, and obtain exponential convergence\footnote{In the optimization literature this is called \emph{linear convergence}, since the decrease of the objective function at a given step is linear in the current sub-optimality.}.

Let's assume that our selection rule selects coordinates $i_t$ with at least a fraction $\beta \in (0,1]$ of the full gradient $\ell_2$ norm squared:
\begin{equation}
    |\partial_{i_t}f(x_t)|^2 \geq \beta \Vert \nabla f(x_t)\Vert_2^2
\label{ass:big_grad}
\end{equation}
For random update rules, this should hold in expectation.
If $i_t$ satisfies this condition, we call it a \emph{good} update. 
To guarantee exponential convergence, we need that the number of good updates grows linearly with $t$, i.e. the density of good updates is non-vanishing. Let $\mathcal{G}(t)$ be the number of good updates up to iteration $t$, then the long-run density of the good updates is  

\begin{equation}
    \liminf_{t\to\infty} \frac{\mathcal{G}(t)}{t} = \rho > 0.
\label{eq:good_dens}
\end{equation}

With the Gauss-Southwell rule, i.e. updating the coordinate with greatest absolute value gradient, we can guarantee that least $\beta \geq 1/n$ for every update. 
Another common rule is updating a uniformly random coordinate, which achieves this in expectation $\E[\beta] = 1/n$.

The following theorem is a standard result in optimization theory (see, e.g.,\cite[Thm.~1]{wright2015coordinatedescentalgorithms}).

\begin{theo}
Let $f$ be $\mu$-PL and admit coordinate-wise Lipschitz constants $L_i$. Consider the coordinate descent method \eqref{eq:coord_desc} with $\alpha_t = 1/L_{i_t}$ and using a selection rule that satisfies \eqref{ass:big_grad}. Then
$$
f(x_{t+1}) - f^* \leq \left(1 - \frac{\beta\mu}{L_{i_t}}\right)(f(x_t)-f^*).
$$
\label{thm:exp_conv}
\end{theo}

With assumption \eqref{eq:good_dens} we obtain exponential convergence by iterating Theorem \ref{thm:exp_conv}. We know that by time $t$ we have at least $\rho t$ updates for which the theorem applies, so we obtain a rate at least
$$
-\rho \log\!\left(1-\frac{\beta\mu}{L_{\max}}\right).
$$

\medskip

We can extend Theorem~\ref{thm:exp_conv} to \emph{block descent}, in which we update a subset of coordinates $B_t \subseteq [n]$ simultaneously at each iteration:

\begin{equation}
    x_{t+1} = x_t - \alpha_t \odot \nabla f(x_t) I_{B_t},
\label{eq:block_desc}
\end{equation}
where $\odot$ denotes the Hadamard product and $\alpha_t$ is a row vector, allowing for non-homogeneous step sizes.

For each block $B \subseteq [n]$, we define the \emph{block Lipschitz constant} $L_B$ as
$$
\Vert \nabla_B f(x + hI_B) - \nabla_B f(x) \Vert_2 \leq L_B \Vert h I_B\Vert_2 \qquad \forall x,h \in \R^n.
$$

In the case of $f(x) = xAx^\top$ with $A$ positive semi-definite, the block Lipschitz coefficients are simply 
$$
L_B = \lambda_{max}(A_{BB}).
$$

\begin{theo}
Suppose $f$ satisfies the PL-inequality with parameter $\mu$ and admits block Lipschitz constants $\{L_B\}$. Consider a block descent method with $\alpha_t = 1/L_{B_t}$ and where the selected block $B_t$ satisfies the block version of \eqref{ass:big_grad}, so $\beta = \frac{  \Vert \nabla f(x) I_B \Vert_2^2}{\Vert \nabla f(x) \Vert_2^2}$. Then
$$
f(x_{t+1}) - f^* \leq \left(1 - \frac{\beta\mu}{L_{B_t}}\right)(f(x_t)-f^*).
$$
\end{theo}

\section{Coordinate Descent for Reversible Chains}
\label{sec:cd_rev}

\subsection{Finding an energy function for RLGL}

We would like to interpret the RLGL update \eqref{eq:rlgl_update} as a gradient descent step for some energy function. This means that the residue vector $r(x)=x(I-P)$ should be a gradient field. It's clear that this is not possible in general, since the Jacobian of a gradient field needs to be symmetric, and this is true if and only if $P$ itself is symmetric. 

To recover a gradient interpretation, we look for an appropriate coordinate transformation. Since the RLGL update is linear, the transformation must also be linear. Hence we seek an invertible matrix $M \in \mathbb{R}^{n \times n}$ such that, in the new coordinates $y = xM$, 
the update becomes the gradient of a quadratic energy. 

\begin{lemma}
The RLGL dynamics can be represented as gradient descent on a quadratic energy if and only if $P$ is similar to a symmetric matrix. 
\end{lemma}

\begin{proof}
Suppose $y = xM$ with $M$ invertible. 
Then in the $y$--coordinates the residue becomes
$$
\tilde r(y) = r(yM^{-1})M = y(I - M^{-1} P M).
$$
If $M^{-1}PM$ is symmetric, then $\tilde r(y)$ is a gradient field with potential
\[
\Phi(y) = \tfrac{1}{2} y (I - M^{-1} P M) y^\top.
\]
Thus RLGL admits a gradient interpretation when $P$ is similar to a symmetric matrix.
\end{proof}

In particular, if $P$ is reversible with stationary distribution $\pi$, then the linear transformation
$$
M = \Pi^{-1/2}, \qquad \Pi = \mathrm{diag}(\pi),
$$
yields $\Pi^{1/2} P \Pi^{-1/2}$, which is symmetric by detailed balance \eqref{eq:DB}. In this case, the quadratic energy is the Dirichlet energy \eqref{eq:dir_energy}.


\subsection{RLGL as block descent}

As established above, for reversible chains, the coordinate transformation $y=x\Pi^{-1/2}$ allows us to rigorously interpret the RLGL dynamics as a gradient descent process on the Dirichlet energy. We can now use this optimization framework to formally connect RLGL to block coordinate descent and analyze its convergence.

\begin{theo}
    Let $P$ be an irreducible, reversible transition matrix. If the updated block $B_t$ is an independent set, the RLGL update matches a block descent update with an optimal step size in the symmetrized coordinates.
\label{thm:equiv}
\end{theo}

\begin{proof}
Let $y = x\Pi^{-1/2}$. In these coordinates, the associated energy is the Dirichlet energy $\En(y) = \tfrac12\, y \mathcal{L} y^\top$,
whose gradient is $\nabla_y \En(y) = y\mathcal{L}.$ Furthermore, the rescaled residual is given by  $\tilde r = -y\mathcal{L}$, so that $\tilde r=-\nabla_y \En(y)$.

Consider any update $u$ supported on $B_t$ (so $u_i=0$ for $i\notin B_t$). The new iterate is $y_{t+1} = y_t + u$, and the energy after the update reads
$$
\En(y_{t+1}) 
= \tfrac12 (y_t+u)\mathcal{L}(y_t+u)^\top
= \En(y_t) - \tilde r_{B_t} u_{B_t}^\top + \tfrac12 u_{B_t}\, \mathcal{L}_{B_tB_t}\, u_{B_t}^\top,
$$
where $\tilde r_{B_t}$ is the restriction of $\tilde r$ to $B_t$ and $\mathcal{L}_{B_tB_t}$ is the principal submatrix of $\mathcal{L}$ on $B_t$. 
This expression is a strictly convex quadratic in $u_{B_t}$, whose unique minimizer gives the optimal block update. The first-order condition is

$$
- \tilde r_{B_t} + u_{B_t}\,\mathcal{L}_{B_tB_t} = 0,
$$
so the optimal block update is
$$
u_{B_t}^* \;=\; \tilde r_{B_t}\, \mathcal{L}_{B_tB_t}^{-1}.
$$

Since $B_t$ is an independent set and the chain is irreducible, $\mathcal{L}_{B_tB_t}=I$, so the optimal update simplifies to $u_{B_t}^* = \tilde r_{B_t}$, meaning the new iterate is obtained simply by adding the residual on the updated coordinates: 
$$
y_{t+1} = y_t + \tilde r_tI_{B_t}.
$$
Comparing this to \eqref{eq:block_desc}, we see that the per-coordinate step size is exactly $1$.

Transforming back into $x$ coordinates, we obtain
$$
x_{t+1}\Pi^{-1/2} = x_t\Pi^{-1/2} - (x_t\Pi^{-1/2})\mathcal{L}I_{B_t}
= x_t\Pi^{-1/2} - x_t(I-P)\Pi^{-1/2}I_{B_t}.
$$
Multiplying on the right by $\Pi^{1/2}$ yields
$$
x_{t+1} = x_t - x_t(I-P)I_{B_t},
$$
which is exactly the RLGL update \eqref{eq:rlgl_update}.
\end{proof}

Since with these assumptions RLGL is a block descent method and the Dirichlet energy \eqref{eq:dir_energy} satisfies the hypothesis of Theorem \ref{thm:exp_conv}, we obtain exponential convergence.

\begin{theo}
If $P$ is reversible with Poincaré constant $\mu$, then one $\beta$-good update yields
\begin{equation}
\En(y_{t+1}) \leq \left(1-\frac{\beta\mu}{L_{B_t}}\right)\En(y_t).
\label{eq:contraction}
\end{equation}
In particular, if the update sequence uses independent sets and has a fraction $\rho$ of $\beta$-good updates, then RLGL converges exponentially, with a rate of at least 
$$
-\rho \log{ \left(1 - \beta \mu \right) }.
$$
\label{thm:rlgl_rev}
\end{theo}

\begin{remark}
If the updated block $B_t$ is not an independent set, then $\mathcal{L}_{B_tB_t}$ is not the identity. 
In this case, RLGL still decreases the energy, but the step size is generally suboptimal. 
For example, if $B_t=\{i\}$ with $P_{ii}>0$, i.e., $i$ has a self-loop, then the optimal step size on coordinate $i$ would be $(1-P_{ii})^{-1}$ rather than $1$. 
Thus exponential convergence is preserved, but the guaranteed per-step decrease is smaller.
\end{remark}

From Theorem~\ref{thm:equiv} we have a new interpretation of RLGL: one update minimizes the Dirichlet energy in the subspace of the updated coordinates when the updated block is an independent set. This makes it look like an Orthogonal-subspace method \cite{saad2003iterative}.

 When the updated block $B_t$ is an independent set, the optimal block update yields an exact energy decrease of
\begin{equation}
  \En(y_t)-\En(y_{t+1}) \;=\; \tfrac12\|\tilde r_{B_t}\|_2^2.  
  \label{eq:en_dec}
\end{equation}
so the Dirichlet energy strictly decreases whenever the block residual is nonzero. 
Crucially, this statement holds irrespective of whether there is any local residual cancellation of the original "raw" residual vector in the $x$-coordinates.


\subsection{When coordinate descent outperforms power iteration}

We now compare the exponential convergence rates of coordinate descent (CD) from Theorem~\ref{thm:rlgl_rev} and power iteration (PI). Since our focus is on computational efficiency, we measure progress per unit of arithmetic cost. Updating a single coordinate of the residual in CD requires work proportional to the out-degree of the chosen node, so that $n$ coordinate updates correspond roughly to one matrix–vector multiplication in PI. This normalization is exact on average when the coordinate is selected uniformly at random.

Recall that PI converges at a rate governed by the second eigenvalue of $P$: the $\ell_2$–error decays as $\lambda_2^t$, or equivalently at an exponential rate of $-\log \lambda_2$. For CD, Theorem~\ref{thm:exp_conv} guarantees that after $n$ coordinate updates the Dirichlet energy decreases at a rate of at least
$$
-n \rho \log\!\bigl(1-\beta(1-\lambda_2)\bigr).
$$
Under uniform random selection, we have $\rho=1$ and $\E[\beta] = 1/n$. We use Proposition~\ref{prop:residual-energy} to compare the two.

In the asymptotic regime $n \gg 1$, CD with the Gauss-Soutwell or random selection is always slower than PI, since
$$
-\log \lambda_2 \;\geq\frac{1-\lambda_2}{2} \qquad \text{for all } \lambda_2 \in [0,1].
$$

However, CD can outperform PI if the coordinate rule achieves a larger $\beta$. Specifically, for $\rho=1$, CD beats PI whenever
$$
\beta \;\geq\; \frac{1-\lambda_2^{2/n}}{1-\lambda_2} \;=\; \frac{2}{n} + \frac{(n-2)(1-\lambda_2)}{n^2} + o\bigl((1-\lambda_2)^2\bigr).
$$
Thus, surpassing PI requires that the chosen coordinate captures more than the average $2/n$ share of the total residual. 

The Gauss--Southwell rule (after the first iteration) achieves a $\beta$ is strictly greater than $1/n$. Empirically, we consistently observe that the selected coordinate accounts for more than $2/n$ of the total residual. The same empirical observation was noted in \cite{OPIC}, where the residual vector is non-negative. 
\begin{figure}[H]
    \centering
    \includegraphics[width=0.65\linewidth]{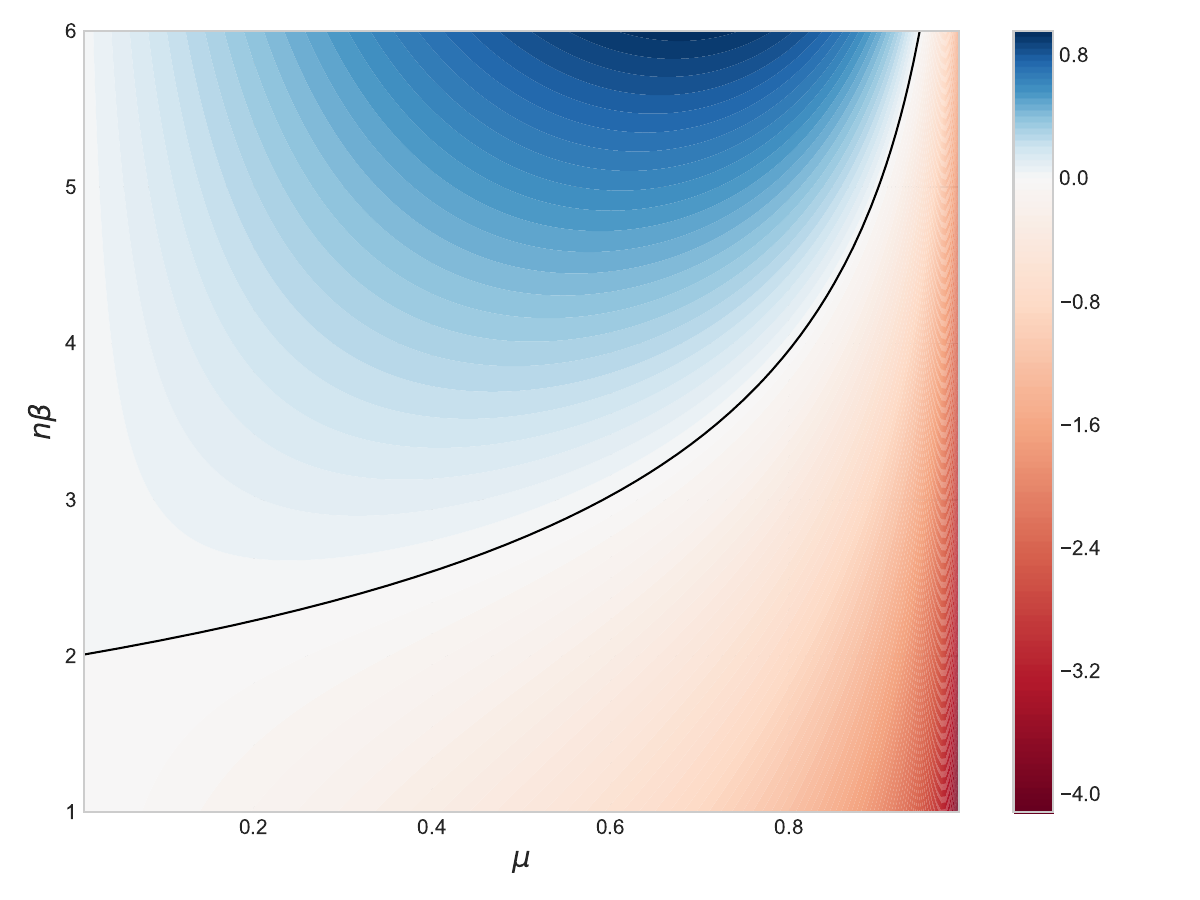}
    \caption{Comparison of convergence rates of Coordinate Descent and Power iteration, for $n=100$. Plotted difference $r_{CD} - r_{PI}$. The blue region is where $n$ iterations of CD beat one PI, in red the converse. The black line shows when $r_{CD} = r_{PI}$.}
    \label{fig:CDvsPI}
\end{figure}

These results suggest a natural interpretation: to surpass PI, the algorithm must avoid diffusing the residual evenly across all coordinates. Instead, it is advantageous to let the residual accumulate so that a few coordinates carry a significant portion of the error, allowing each update to make more substantial progress. In sparse networks this concentration occurs naturally, which explains why CD can outperform PI in practice even when theoretical upper bounds based on the spectral gap $\mu$ are pessimistic. 

\section{Extension to Nearly Reversible Chains}
\label{sec:nearly_rev}
In this section we extend Theorem~\ref{thm:exp_conv} beyond the reversible setting. Our goal is to identify conditions, under which exponential convergence can still be guaranteed for general, possibly non-reversible, Markov chains. To this end, we decompose an arbitrary transition matrix into a reversible component and a perturbation term.

We then analyze the behavior of coordinate-descent dynamics in the presence of this linear perturbation, and derive conditions, under which exponential convergence is preserved. When specialized to our setting, these conditions involve the row norm of the irreversible part of the transition matrix and the Poincar\'e constant of the chain. As we will show, for single-coordinate update rules—such as the Greedy Gauss--Southwell and the Random rule—to guarantee a one-step decrease in the energy, the irreversibility of the chain must be vanishingly small, in a precise sense described below.


\subsection{Irreversible chains as perturbed coordinate descent}

We now relax the reversibility assumption. In this setting, convergence of the RLGL updates is no longer guaranteed in general.  
For example, on the directed cycle one can keep pushing the residual along the cycle indefinitely without ever reducing its norm, and even some aperiodic chains exhibit this phenomenon~\cite{RLGL}.  
Thus, in the irreversible setting we must impose additional structural assumptions on the chain, the update sequence, and their interaction.

Our approach is to interpret RLGL as coordinate descent with a perturbed gradient.  
The perturbation arises from the antisymmetric part of the dynamics, and we will derive spectral conditions, under which it remains small enough to preserve convergence.

Let $P$ be an arbitrary (not necessarily reversible) Markov chain with stationary distribution $\pi$.  
Define the symmetric and antisymmetric components of $P$ in the $\pi$-weighted coordinates by
$$
S = \tfrac{1}{2}\,\Pi^{1/2}(P + P^*)\Pi^{-1/2},
\qquad
A = \tfrac{1}{2}\,\Pi^{1/2}(P - P^*)\Pi^{-1/2},
$$
so that $S^\top = S$ and $A^\top = -A$.  
With this notation, the normalized Laplacian decomposes as
$$
\mathcal{L}_{\mathrm{norm}}
= \Pi^{1/2}(I-P)\Pi^{-1/2}
= (I-S) - A
= \L - A,
$$
Writing the updates in normalized coordinates $y_t = x_t \Pi^{-1/2}$, the RLGL rule \eqref{eq:rlgl_update} becomes
\begin{equation}
    y_{t+1}
    = y_t - y_t\bigl[(I-S) - A\bigr] I_{B_t}
    = \underbrace{y_t - y_t\L I_{B_t}}_{\text{descent step}}
      \;+\;
      \underbrace{y_t A I_{B_t}}_{\text{perturbation}}.
\label{eq:pert_up}
\end{equation}
The first term is a coordinate descent step for the Dirichlet energy, while the second term is the
perturbation induced by the irreversible dynamics.
If this perturbation remains sufficiently small, in a sense made precise below, RLGL still exhibits exponential convergence.

To quantify how much each coordinate deviates from reversibility, we introduce the following \emph{local irreversibility} coefficients for each state:
\begin{equation}
    \kappa_i
    := \|A_{i,:}\|_2
    = \frac{1}{2\sqrt{\pi_i}}
        \left( \sum_{j=1}^n
               \frac{1}{\pi_j}\,(\pi_j p_{ji} - \pi_i p_{ij})^2
        \right)^{1/2}.
\label{eq:kappa_i}
\end{equation}
The quantity $\kappa_i$ captures the local irreversibility at state $i$. It measures how strongly the chain violates detailed balance at $i$, by weight averaging the squared net probability currents with respect to the stationary distribution. When $\kappa_i = 0$, state $i$ satisfies detailed balance with all other states. Thus, $\kappa_i$ quantifies the extent to which the dynamic around $i$ fails to be reversible.

We now introduce two global irreversibility coefficients: one based on the $\ell_\infty$ norm and the other based on the $\ell_2$ norm.
\begin{equation}
\label{eq:eta}
\eta_\infty := \max_{i\in[n]}\frac{\kappa_i}{\mu}, \qquad \eta_2 := \frac{1}{\mu} \left(\sum_{i=1}^n \kappa_i^2\right)^{1/2},
\end{equation}
where $\mu$ is the Poincaré constant of $P$. Note that we have the bound $\eta_2 \leq \sqrt{n}\, \eta_\infty$.

\begin{defn}[Nearly reversible chains]
We say that a Markov chain with transition matrix $P$ is \emph{nearly reversible} if
$$
\eta_\infty < \frac{1}{2n + \sqrt{n}},
$$
\label{def:almost_rev}
\end{defn}

This threshold is chosen because it is precisely the regime in which single–coordinate
update rules remain effective: as we will see in Corollary~\ref{cor:random_greedy_explicit},
such a bound on $\eta_\infty$ guarantees that both the Greedy Gauss--Southwell rule and the
Random rule produce a one-step decrease in energy and therefore achieve exponential convergence.



\subsection{Coordinate descent with linear perturbation}

As observed in \eqref{eq:pert_up}, an RLGL update can be decomposed into a descent step plus an additional perturbation term. This perturbation is linear in the current iterate and arises from the irreversible component of the transition matrix \eqref{eq:rev_dec}. If this perturbation is sufficiently small, one may still expect convergence.

To make this precise, we analyze a generic coordinate descent update in the presence of a linear perturbation and derive explicit conditions under which convergence is preserved. Additive perturbations for full gradient descent have been studied in prior work (see, e.g., \cite{gd_perturb}); here we adapt the core argument to the coordinate-wise setting.

Motivated by the updates \eqref{eq:pert_up}, we will study the case when, for each coordinate $i$, the perturbation is a linear function of the current iterate:
$$
\xi_i(x_t) = \bigl\langle x_t, v^{(i)}\bigr\rangle,
$$
for some fixed $v^{(i)}\in\mathbb{R}^n$. Set  $\kappa_i := \Vert v^{(i)}\Vert_2$ to be  $\ell_2$-norm of these vectors, and let $\kappa := \max_i \kappa_i$. (We will show in Section~\ref{sec:nearly-reversible} that these are the same $\kappa$'s as defined in \eqref{eq:kappa_i}.)

At iteration $t$ we pick coordinate $i_t$ and perform the update
\begin{equation}
x_{t+1}
= x_t
- \alpha_t\,\partial_{i_t}f(x_t)\,e_{i_t}
- \xi_{i_t}(x_t)\,e_{i_t}.
\label{eq:pert_cd}
\end{equation}
We impose a natural consistency condition that the perturbation must vanish at the minimizer:
\begin{equation}\label{ass:pert_consistency}
\langle v^{(i)}, x^*\rangle = 0,\qquad \forall i\in[n],
\end{equation}
which ensures that $x^*$ remains a fixed point of the dynamics \eqref{eq:pert_cd}. The next theorem gives a  sufficient condition for exponential convergence of perturbed coordinate descent.

\begin{theo}
\label{thm:perturbed_global}
Let $f:\mathbb{R}^n\to\mathbb{R}$ be differentiable and $\mu$-strongly convex, and assume each partial derivative $\partial_i f$ is $L_i$-Lipschitz. Let $L_{\max} := \max_i L_i$. Suppose the selection rule guarantees updates that are \emph{good} with parameter $\beta\in(0,1]$ (as above). Let the perturbations satisfy the above Assumption \ref{ass:pert_consistency}.

If
\begin{equation}
    \kappa < \frac{\mu\sqrt{\beta}}{L_{\max}},
\label{eq:kappa_cond}
\end{equation}
then for a fixed step size $\alpha_t = 1/L_{i_t}$, the iterates converge exponentially to $x^*$. Specifically, for any good update $i_t$, the per-step decrease is at least:
$$
f(x_{t+1}) - f^*\leq \left(1 - \frac{\beta\mu}{L_{i_t}} + \frac{L_{i_t}\kappa_{i_t}^2}{\mu}\right)(f(x_t) - f^*),
$$
leading to an overall rate of at least
$$
-\log \left( 1 - \frac{\beta\mu}{L_{max}} + \frac{\kappa^2}{\mu} \right).
$$
\end{theo}
\begin{proof}
The update along coordinate $i_t$ is $\Delta_t = -\alpha_t \partial_{i_t}f(x_t) + \xi_{i_t}(x_t)$. Lemma \ref{lemma:descent} gives
$$
f(x_{t+1}) \leq f(x_t) + \partial_{i_t}f(x_t)\Delta_t + \frac{L_{i_t}}{2}\Delta_t^2.
$$
We choose the standard fixed step size $\alpha_t = 1/L_{i_t}$. Substituting this into the update $\Delta_t$ and then into the inequality above yields:
\begin{align*}
f(x_{t+1}) &\leq f(x_t) - \partial_{i_t}f(x_t)\left(\frac{1}{L_{i_t}}\partial_{i_t}f(x_t) - \xi_{i_t}(x_t)\right) + \frac{L_{i_t}}{2}\left(\frac{1}{L_{i_t}}\partial_{i_t}f(x_t) - \xi_{i_t}(x_t)\right)^2 \\
&= f(x_t) - \frac{1}{L_{i_t}}|\partial_{i_t}f(x_t)|^2 - \partial_{i_t}f(x_t)\xi_{i_t}(x_t) + \frac{1}{2L_{i_t}}|\partial_{i_t}f(x_t)|^2 + \partial_{i_t}f(x_t)\xi_{i_t}(x_t) + \frac{L_{i_t}}{2}|\xi_{i_t}(x_t)|^2 \\
&= f(x_t) - \frac{1}{2L_{i_t}}|\partial_{i_t}f(x_t)|^2 + \frac{L_{i_t}}{2}|\xi_{i_t}(x_t)|^2.
\end{align*}
Now, we bound the last two terms. Since $f$ is $\mu$-strong convex, by Lemma~\ref{lem:stronly-convex} $f$ is also $\mu$-PL. Thus for a $\beta$-good update \eqref{ass:big_grad}, we obtain
$$
|\partial_{i_t} f(x_t)|^2 \geq 2\beta\mu (f(x_t)-f(x^*)).
$$
For the perturbation, we use the $\mu$-strong convexity of $f$ and the consistency condition \eqref{ass:pert_consistency}:
$$
|\xi_{i_t}(x_t)|^2 = |\langle x_t, v^{(i_t)}\rangle|^2 = |\langle x_t-x^*, v^{(i_t)}\rangle|^2 \leq \kappa_{i_t}^2\Vert x_t-x^*\Vert_2^2 \leq \frac{2\kappa_{i_t}^2}{\mu}(f(x_t)-f(x^*)).
$$
Substituting these bounds into our descent inequality, we get:
$$
f(x_{t+1}) \leq f(x_t) - \frac{1}{2L_{i_t}}\left(2\beta\mu(f(x_t)-f^*)\right) + \frac{L_{i_t}}{2}\left(\frac{2\kappa_{i_t}^2}{\mu}(f(x_t)-f^*)\right).
$$
Subtracting the optimal value $f^*$ from both sides and rearranging gives the one-step contraction:
$$
f(x_{t+1})-f^* \leq \left(1-\frac{\beta\mu}{L_{i_t}} + \frac{L_{i_t}\kappa_{i_t}^2}{\mu}\right)(f(x_t)-f^*).
$$
Condition \eqref{eq:kappa_cond} ensures that this is a contraction. By considering the worst-case values $L_{\max}$ and $\kappa$, we obtain the desired exponential convergence rate.
\end{proof}


Computing the optimal step size $\alpha_t^*$ for the perturbed update yields
\begin{equation}
    \alpha_t^* = \frac{1}{L_{i_t}} + \frac{\xi_{i_t}(x_t)}{\partial_{i_t}f(x_t)},
\label{eq:pert_alpha}
\end{equation}
which is the same as the unperturbed case plus a the ratio of the perturbation and the partial derivative with respect to the updated coordinate. Thus if the latter is small, the unperturbed step size is close to the optimal one.

\subsection{RLGL on nearly reversible chains}\label{sec:nearly-reversible}

Our goal is to show that if the Markov chain is nearly reversible, and the update sequence satisfies a condition of the form \eqref{ass:big_grad}, then the RLGL algorithm converges exponentially. We achieve this by applying Theorem~\ref{thm:perturbed_global}, since in the irreversible case we can view the RLGL update as perturbed coordinate descent \eqref{eq:pert_up}.

The perturbation term in \eqref{eq:pert_up} takes the form 
$$
\xi_{i_t}(y_t) = (y_tA)e_{i_t} = \langle y_t, A_{:,i_t}\rangle,
$$
i.e., the perturbation vector is the $i_t$-th column of the anti-symmetric matrix $A$, so $v^{(i)} = A_{:,i}$, meaning $\kappa_i =  \Vert A_{:,i} \Vert_2$, which is precisely our measure of local irreversibility \eqref{eq:kappa_i}.

We define two residuals:
$$
s_t = y_t\mathcal{L}_\text{sym}, \qquad r_t=y_t\mathcal{L}_{{\rm norm}}.
$$
The residual $s_t$ corresponds to the gradient of the Dirichlet energy, and is the quantity that appears in the convergence theorems for coordinate descent. However, the residual actually observed in RLGL during execution is $r_t$. In the reversible case the two coincide, but in the irreversible setting this distinction becomes essential: the update sequence must be 
$\beta-$good with respect to $s_t$, while selection happens using $r_r$.
We therefore define
$$
|s_{t,i}|^2\geq \beta_\text{sym}\Vert s_t\Vert_2^2 , \qquad |r_{t,i}|^2\geq \beta_{\rm norm}\Vert r_t\Vert_2^2. 
$$
Later, in Proposition~\ref{prop:beta_transfer}, we will quantify how far the two residuals can differ under the nearly-reversible assumption and show how we can relate $\beta_\text{sym}$ and $\beta_{\rm norm}$.

We now prove a bound for the perturbation term, which we will need to apply Theorem~\ref{thm:perturbed_global} to the updates based on $r_t$. Recall that the Dirichlet energy \eqref{eq:dir_energy} is not strongly convex, but it is $\mu$-PL.

\begin{lemma}
    For each coordinate $i$,
    $$
    |(yA)_i|^2 \leq \frac{2\kappa_i^2}{\mu}\En(y), \qquad \forall y\in \R^n.
    $$
\label{lemma:perturb_bound}
\end{lemma}
\begin{proof}
Let $y^\perp$ be the projection of $y$ in $Ker(\L)^\perp$. By the consistency condition \eqref{ass:pert_consistency} and Cauchy-Schwarz,
$$
|\xi_i(y)| = |\xi_i(y^\perp)| \leq \kappa_i \Vert y^\perp \Vert_2.
$$
In $Ker(\L)^\perp$ the Dirichlet energy is $\mu$-strongly convex, and the unique minimizer is $0$, hence
$$
\En(y^\perp) \geq \frac{\mu}{2}\Vert y^\perp\Vert_2^2.
$$
Noting that $\En(y)=\En(y^\perp)$, we obtain
$$
|\xi_i(y)|^2 \leq \frac{2\kappa_i^2}{\mu} \En(y).
$$
\end{proof}

\begin{theo}
\label{thm:rlgl_conv}
Let $P$ be a Markov chain with Poincaré constant $\mu$. Then one $\beta_\text{sym}$-good update yields:
$$
\En(y_{t+1}) \leq \left(1 - \frac{\beta_\text{sym}\mu}{L_{i_t}} + \frac{L_{i_t}\kappa_{i_t}^2}{\mu}\right)\En(y_t).
$$
If $P$ is such that
\begin{equation}
\eta_\infty < \frac{\sqrt{\beta_\text{sym}}}{L_{max}},
\label{eq:eta_cond}
\end{equation}
and $P$ has no self loops ($p_{ii} = 0$ for all $i$), and the update sequence $(i_t)_{t\geq0}$ has only $\beta_\text{sym}$-good updates, then the RLGL algorithm converges exponentially with a rate of at least
$$
-\log \left( 1 - \beta_\text{sym}\,\mu + \frac{\kappa^2}{\mu} \right).
$$
\end{theo}

\begin{proof}
The only step in the proof of Theorem~\ref{thm:perturbed_global} that required $\mu$-strong convexity is the bound on the perturbation term. In our case, we can use Lemma~\ref{lemma:perturb_bound}. 

The condition for convergence from Theorem \ref{thm:perturbed_global} is $\kappa < \mu\sqrt{\beta_\text{sym}}/L_{\max}$. This condition is ensured by the definition of $\eta_\infty=\kappa/\mu$ in \eqref{eq:eta} and \eqref{eq:eta_cond}. 
Applying Theorem~\ref{thm:perturbed_global}, for a good update we have a one-step contraction of at least:
$$
\En(y_{t+1}) \leq \left(1 - \frac{\beta_\text{sym}\:\mu}{L_{i_t}} + \frac{L_{i_t}\kappa_{i_t}^2}{\mu}\right)\En(y_t).
$$
The Lipschitz constants are $L_i = 1-p_{ii}$. For a chain with no self-loops, $p_{ii}=0$, so $L_i=1$ for all $i$, and $L_{\max}=1$.
Substituting $L_{i_t}=1$ and taking the worst-case bounds over all coordinates we obtain the stated convergence rate.
\end{proof}

Computing the energy decrease per-step in the case of no self-loops, we find:
$$
\En(y_t)-\En(y_{t+1}) = \frac{1}{2}|s_t|^2 - \frac{1}{2}|\xi_{i_t}(x_t)|^2,
$$
which shows that if the perturbation is smaller than the residual, the energy decreases.

Our objective is now to guarantee that a $\beta_{\rm norm}$-good coordinate remains $\beta_\text{sym}$-good for some positive, possibly smaller, $\beta_\text{sym}$. To achieve this, we need to compare the two residuals.

\begin{lemma}
For every coordinate $i$,
$$
\frac{|r_{t,i} - s_{t,i}|}{\|s_t\|_2} \le \eta_\infty, \qquad \frac{\Vert r_{t} - s_{t}\Vert_2}{\|s_t\|_2} \le \eta_2.
$$
\end{lemma}

\begin{proof}
From Lemma~\ref{lemma:perturb_bound} we have
$$
|r_{t,i} - s_{t,i}|^2
= |(y_t A)_i|^2
\;\le\; \frac{2\kappa_i^2}{\mu}\,\mathcal{E}(y_t).
$$
Using the PL inequality \eqref{def:pl_ineq},
$$
\mathcal{E}(y_t) \le \frac{1}{2\mu}\|s_t\|_2^2,
$$
so we obtain
$$
\frac{|r_{t,i} -  s_{t,i}|^2}{\| s_t\|_2^2}
\le \frac{\kappa_i^2}{\mu^2}.
$$
Using the definition of $\eta_\infty$ and taking square roots yields the first bound. Summing over the coordinates yields the second.
\end{proof}

\begin{prop}
\label{prop:beta_transfer}
Let $P$ be nearly reversible. If a coordinate $i$ is $\beta_{\rm norm}$-good, such that
\begin{equation}
\label{eq:beta_alpha_star_inf}
  \gamma^\star_\infty := \frac{\sqrt{n}\,\eta_\infty}{1-\sqrt{n}\,\eta_\infty}  \leq \sqrt{\beta_{\rm norm}},
\end{equation}
then the same coordinate is $\beta_\text{sym}$-good with
$$
\beta_{\mathrm{sym}}
\ge \left(\frac{\sqrt{\beta_{\rm norm}}-\gamma^\star_\infty}{1+\gamma^\star_\infty}\right)^2.
$$
\end{prop}

\begin{proof}
Define $\delta_t = r_t - s_t = y_t A$. From the coordinate-wise bound of Lemma~\ref{lemma:perturb_bound} and the definition $\kappa:=\max_i \kappa_i$ we get
$$
\|\delta_t\|_2 \le \frac{\sqrt{n}\,\kappa}{\mu}\,\|s_t\|_2 \le \sqrt{n}\,\eta_\infty\,\|s_t\|_2.
$$
Set $\gamma_t := \|\delta_t\|_2/\|r_t\|_2$. We have
$$
\|r_t\|_2 = \|s_t+\delta\|_2 \ge \|s_t\|_2 - \|\delta_t\|_2 \ge (1-\sqrt{n}\,\eta_\infty)\|s_t\|_2,
$$
and hence
$$
\gamma_t = \frac{\|\delta_t\|_2}{\|r_t\|_2} \le \frac{\sqrt{n}\,\eta_\infty}{1-\sqrt{n}\,\eta_\infty} = \gamma^\star_\infty.
$$
Furthermore, by Definition~\ref{def:almost_rev}, it holds that $0<\gamma_t\leq 1$.

Since by assumption the update is $\beta_{norm}$-good, we have $|r_{t,i}|\ge \sqrt{\beta_{\rm norm}}\,\|r_t\|_2$, hence
$$
|s_{t,i}| \ge |r_{t,i}| - |\delta_i| \ge (\sqrt{\beta_{\rm norm}}-\gamma_t)\|r_t\|_2.
$$
Also $\|s_t\|_2 \le \|r_t\|_2 + \|\delta\|_2 = (1+\gamma_t)\|r_t\|_2$. Therefore
$$
\frac{|s_{t,i}|^2}{\|s_t\|_2^2} \ge \frac{(\sqrt{\beta_{\rm norm}}-\gamma_t)^2}{(1+\gamma_t)^2}.
$$
The right-hand side is decreasing in $\gamma_t$ for $0\leq \gamma_t\leq \sqrt{\beta_{\rm norm}}$, so substituting $\gamma_t=\gamma^\star_\infty$ yields the stated lower bound for $\beta_{\rm sym}$.
\end{proof}

We can obtain an analogous bound using $\eta_2$.

\begin{prop}
\label{prop:beta_transfer_l2}
Let $P$ be nearly reversible. If a coordinate $i$ is $\beta_{\rm norm}$-good, such that
\begin{equation}
\label{eq:beta_alpha_star_l2}
  \gamma^\star_2 := \frac{\eta_2}{1-\eta_2}  \leq \sqrt{\beta_{\rm norm}},
\end{equation}
then the same coordinate is $\beta_\text{sym}$-good with
$$
\beta_{\mathrm{sym}}
= \left(\frac{\sqrt{\beta_{\rm norm}}-\gamma^\star_2}{1+\gamma^\star_2}\right)^2.
$$
\end{prop}
\begin{proof}
We start by summing up the inequalities of Lemma~\ref{lemma:perturb_bound}. The rest of the proof follows.
\end{proof}

\begin{theo}
\label{thm:rlgl_nearly_rev_combined}
Let $P$ be nearly reversible. If the update schedule satisfies
\begin{equation}
    \eta_\infty^2 < \left(\frac{\sqrt{\beta_{\rm norm}}-\gamma^\star}{1+\gamma^\star}\right)^2 =: \tilde \beta,
\label{eq:beta_tilde}
\end{equation}
where $\gamma^\star$ is defined as in \eqref{eq:beta_alpha_star_inf} or \eqref{eq:beta_alpha_star_l2}. 
Then the RLGL algorithm converges exponentially, with a rate of at least
$$
-\log\left(1-\tilde \beta\mu + \frac{\kappa^2}{\mu}\right)
$$
\end{theo}
\begin{proof}
    By Proposition~\ref{prop:beta_transfer} or Proposition~\ref{prop:beta_transfer_l2}, $\beta_\text{sym} \geq \tilde\beta$. We then apply Theorem~\ref{thm:rlgl_conv}.
\end{proof}


\begin{corollary}
\label{cor:random_greedy_explicit}
Let $P$ be nearly reversible. Then the RLGL algorithm converges exponentially for the greedy Gauss-Southwell and random rules.
\end{corollary}

\begin{proof} 
See Appendix~\ref{app:random_greedy_proof}
\end{proof}

We want to highlight that near-reversibiltiy is a strong condition because we require the
irreversibility term to scale as $O(1/n)$. However, under this condition, we also obtain a strong result that the fraction by which the energy is reduced, is $O(1/n)$ as well, while updating only one coordinate. Intuitively, we see the perturbation as noise, and we require the noise to be smaller than the update.

When the chain is not nearly reversible, updating a single coordinate is insufficient: the method cannot guarantee a descent step. Instead, one must update a block of coordinates so that the update rule has a sufficiently large $\beta$ to tolerate the higher irreversibility. This might explain the empirical success of RLGL on irreversible chains: although a single-coordinate update may fail to decrease the energy, multiple or block updates can do so. A purely one-step analysis, which is what we have carried out, is insufficient to capture this phenomenon.

\subsection{Nearly reversible convex combinations}

We analyze the nearly reversible property (Definition~\ref{def:almost_rev}) for convex combination of reversible and irreversible chains. 
Consider
\begin{equation}
\label{eq:mixture}
P_\varepsilon = (1-\varepsilon)R + \varepsilon Q,\qquad \varepsilon\in[0,1],
\end{equation}
where $R$ is irreducible and reversible with respect to $\pi$, and $Q$ is irreversible. 
The local irreversibility constants $\kappa_i$, the Poincar\'e constant $\mu$, and the corresponding $\eta_\infty$ will be functions of $\varepsilon$.

Consider first the simple case where $Q$ has the same stationary distribution $\pi$.
Since $A(R)=0$, the anti-symmetric part satisfies $A(P_\varepsilon)=\varepsilon A(Q)$, so that
$$
\kappa(P_\varepsilon)=\varepsilon\,\kappa(Q).
$$
By Weyl's inequality (see, e.g., \cite[Thm.~4.3.1]{horn_johnson_matrix_analysis}), it follows that
$$
\mu(P_\varepsilon)\ge (1-\varepsilon)\mu(R).
$$
Hence
$$
\eta_\infty(P_\varepsilon)
=\frac{\kappa(P_\varepsilon)}{\mu(P_\varepsilon)}
\le
\frac{\varepsilon}{1-\varepsilon}\,
\frac{\kappa(Q)}{\mu(R)}.
$$
Thus, $P_\varepsilon$ is nearly reversible for $\varepsilon$ small enough.

The restriction that $R$ and $Q$ share the same stationary distribution can be relaxed.

\begin{prop}
\label{prop:pert}
Let $P_\varepsilon$ be of the form \eqref{eq:mixture}. Let $\pi(\varepsilon)$ be the stationary distribution of $P_\varepsilon$. Then there are factors
$$
c_1(\varepsilon):=\frac{(1-\varepsilon)\|R\|_2}{\pi_{\min}(\varepsilon)},\qquad
c_2(\varepsilon):=\frac{\|Q\|_2\,\pi_{\max}(\varepsilon)}{\pi_{\min}(\varepsilon)},
$$
such that
\begin{equation}\label{eq:kappa-const}
\kappa(\varepsilon)\le c_1(\varepsilon)\,\|\pi(\varepsilon)-\pi\|_\infty
+\varepsilon\,c_2(\varepsilon).
\end{equation}
In particular $\kappa(\varepsilon)=O(\varepsilon)$ as $\varepsilon\downarrow0$.
\end{prop}

\begin{proof}
Set
$$
M(\varepsilon):=\Pi(\varepsilon)P_\varepsilon-P_\varepsilon^{\!\top}\Pi(\varepsilon),
\qquad A(\varepsilon)=\tfrac12\,\Pi(\varepsilon)^{-1/2}M(\varepsilon)\Pi(\varepsilon)^{-1/2}.
$$
Since $\| \Pi(\varepsilon)^{-1/2}\|_2=1/\sqrt{\pi_{\min}(\varepsilon)}$ we have
$$
\kappa(\varepsilon)\le\|A(\varepsilon)\|_2
\le\frac{\|M(\varepsilon)\|_2}{2\pi_{\min}(\varepsilon)}.
$$
Using \eqref{eq:mixture} and the reversibility relation
$\Pi R-R^\top\Pi=0$, we expand
$$
\begin{aligned}
M(\varepsilon)
&=(1-\varepsilon)\big(\Pi(\varepsilon)R-R^\top\Pi(\varepsilon)\big)
+\varepsilon\big(\Pi(\varepsilon)Q-Q^\top\Pi(\varepsilon)\big)\\
&=(1-\varepsilon)\big(\Delta(\varepsilon)R-R^\top\Delta(\varepsilon)\big)
+\varepsilon\big(\Pi(\varepsilon)Q-Q^\top\Pi(\varepsilon)\big),
\end{aligned}
$$
where $\Delta(\varepsilon)=\Pi(\varepsilon)-\Pi$ is diagonal. Taking operator norms and using
$\|\Delta(\varepsilon)\|_2=\|\pi(\varepsilon)-\pi\|_\infty$ and $\|\Pi(\varepsilon)\|_2=\pi_{\max}(\varepsilon)$ gives
$$
\|M(\varepsilon)\|_2
\le 2(1-\varepsilon)\|R\|_2\,\|\pi(\varepsilon)-\pi\|_\infty
+2\varepsilon\,\|Q\|_2\,\pi_{\max}(\varepsilon).
$$
Combining with the previous inequality yields
$$
\kappa(\varepsilon)
\le \frac{(1-\varepsilon)\|R\|_2}{\pi_{\min}(\varepsilon)}\|\pi(\varepsilon)-\pi\|_\infty
+\frac{\varepsilon\|Q\|_2\,\pi_{\max}(\varepsilon)}{\pi_{\min}(\varepsilon)},
$$
which is exactly \eqref{eq:kappa-const} with the factors $c_1(\varepsilon), c_2(\varepsilon)$ defined above.

To obtain the $O(\varepsilon)$ claim we use the Schweitzer updating formula \cite{schweitzer1968perturbation}
$$
\|\pi(\varepsilon)-\pi\|_\infty
\le \frac{\varepsilon\,\|(Q-R)D\|_\infty}{1-\varepsilon\,\|(Q-R)D\|_\infty} \quad \text{ whenever } \quad \varepsilon \Vert (Q-R)D\Vert_\infty < 1,
$$
where $D$ is the deviation matrix, given by 
$$
D = (I -R + \1^\top \pi)^{-1} - \1^\top \pi.
$$
Substituting this into \eqref{eq:kappa-const} shows both terms on the right-hand side are $O(\varepsilon)$ as $\varepsilon\downarrow0$, hence $\kappa(\varepsilon)=O(\varepsilon)$.
\end{proof}

As an example, let $Q$ be the transition matrix of the directed cycle on $n$ vertices, and let $R=\tfrac{1}{n}\mathbf{1}^\top\mathbf{1}$. Both $Q$ and $R$ have the uniform stationary distribution $\pi=\tfrac{1}{n}\mathbf{1}$. Consider $P_\varepsilon$ as in \eqref{eq:mixture}.
The antisymmetric part is
\[
A(P_\varepsilon)=\tfrac{1}{2}(P_\varepsilon - P_\varepsilon^\top)
             = \tfrac{\varepsilon}{2}\,(Q - Q^\top),
\]
therefore the local irreversibility coefficients are
\[
\kappa(P_\varepsilon)=\max_i \|A_{i,:}\|_2 = \frac{\varepsilon}{\sqrt{2}}.
\]
The symmetric part is
\[
S(P_\varepsilon)
  = \tfrac{1}{2}(P_\varepsilon + P_\varepsilon^\top)
  = (1-\varepsilon)R + \varepsilon\,\tfrac{Q+Q^\top}{2}.
\]
Its eigenvalues are
\[
\lambda_0(S)=1,\qquad
\lambda_k(S)=\varepsilon \cos\!\frac{2\pi k}{n},\quad k=1,\dots,n-1,
\]
so the Poincaré constant of $S$ is
\[
\mu=1-\varepsilon\cos\!\frac{2\pi}{n}.
\]
Thus the irreversibility ratio is
\[
\eta_\infty
 = \frac{\varepsilon/\sqrt{2}}{\,1-\varepsilon\cos(2\pi/n)\,}.
\]
The threshold $\varepsilon^*(n)$ at which the chain is nearly reversible, i.e. $\eta_\infty \leq 1/(2n+\sqrt{n})$, is
$$
\varepsilon^\ast(n)
  = \frac{1}{(\sqrt{2}\,n + \sqrt{n/2})(1+\cos(2\pi/n)}.
$$
For $n \geq 4$, $\varepsilon^\ast(n) \leq 1/(\sqrt{2}\,n)$, in particular $\varepsilon^\ast(n) = O(1/n)$.
\medskip

This example has direct implications for PageRank: the PageRank transition matrix is precisely of the form~\eqref{eq:mixture}, where the reversible component is the teleportation distribution. Proposition~\ref{prop:pert} shows that by choosing a sufficiently large teleportation factor, one can always ensure that the resulting PageRank chain becomes nearly reversible. The directed-cycle example illustrates a worst-case scenario: in practice, we expect that for many real-world networks the amount of teleportation required to achieve near-reversibility is substantially smaller than what is needed for directed cycles of the same size.

\section{Heuristics and Numerical Experiments}
\label{sec:numerical}
Inspired by the equivalence of RLGL to coordinate descent in the reversible case, we come up with new heuristics. We show that they perform well even in the irreversible case, when we employ them to compute the stationary distribution and PageRank of real world and syntetic networks.

\subsection{Gauss--Southwell--Dirichlet (GSD) heuristics}

The insights of Theorem~\ref{thm:equiv} suggest the following new class of heuristics. Indeed, intuitively, one should prioritize coordinates updates that yield the largest one-step decrease of the Dirichlet energy in \eqref{eq:en_dec}. This suggests heuristics that choose the updated coordinates based on the residual rescaled by $\sqrt{\pi_i}$ rather than the actual residual as was done in \cite{RLGL}, or rescaled by $\pi_i$ as proposed in \cite{mcsherry_uniform_2005}. 

Specifically, assume that $P$ has no self-loops, so that any singleton $\{i\}$ is an independent set. 
By Theorem \ref{thm:equiv}, the one-step decrease in the symmetrized coordinates from updating coordinate $i$ with the optimal step is
\begin{equation}
    \mathcal{E}(y_t) - \mathcal{E}(y_{t+1})
    = \tfrac{1}{2}\big|(y_t\mathcal{L})_i\big|^2.
\label{eq:en_decrease}
\end{equation}
Hence the greedy (Gauss--Southwell) rule in the $y$--coordinates is
\[
i_t = \argmax_{i} \, |(y_t\mathcal{L})_i|.
\]
Translating back to the original coordinates $x$, this becomes
\[
i_t = \argmax_{i} \, \big| \big(x_t(I-P)\Pi^{-1/2}\big)_i \big|,
\]
i.e., one should apply Gauss--Southwell to the residuals \emph{rescaled} by $\sqrt{\pi_i}$. We call this the \emph{Gauss--Southwell--Dirichlet} (\texttt{GSD}) rule.

The stationary distribution $\pi$ is not known apriori, so in practice one must use its proxies $\hat\pi$. A natural choice is to use the current iterate $x_t$. Note that the proxy does not need to be normalized.

\medskip

To exploit parallel hardware, one wants to update many coordinates at each iteration.
Theorem~\ref{thm:equiv} suggests updating an \emph{independent set} that carries large total rescaled residue.  
A simple local scheme that produces such a set is the following:
\begin{enumerate}
  \item each node computes its rescaled residue.
  \item node $i$ is \emph{updated} if it holds the largest rescaled residue in its neighborhood.
\end{enumerate}

Due to the locality, this heuristic is distributed, and yields an independent set $B_t$ that contains the coordinate with maximum rescaled residuals; hence it preserves the descent property from Theorem \ref{thm:equiv}.

The heuristics above are robust to multiplicative approximation errors in the rescaling factor. The following lemma quantifies this statement.

\begin{lemma}
\label{lem:robustness}
Let $\pi$ be the true stationary distribution and let $\hat\pi$ be a proxy satisfying
$$
c^{-1}\,\pi_i \;\le\; \hat\pi_i \;\le\; c\,\pi_i
\qquad\text{for all }i \in [n],
$$
for some $c\geq 1$.  Denote by $\tilde r_i = |(x_t(I-P)\Pi^{-1/2})_i|$ the true rescaled residuals and by $\hat r_i = |(x_t(I-P)\widehat\Pi^{-1/2})_i|$ the residuals rescaled using the proxy. Then
$$
c^{-1/2}\, \tilde r_i \;\le\; \hat r_i \;\le\; c^{1/2}\, \tilde r_i
\qquad\text{for all }i \in [n].
$$
In particular, the coordinate selected by the proxy rule satisfies
$$
\max_i\hat r_i \;\ge\; c^{-1/2} \max_i \tilde r_i,
$$
and the corresponding one-step energy decrease is at least $c^{-1}$ times the optimal single-coordinate decrease.
\end{lemma}

\begin{proof}
Since $\widehat\Pi^{-1/2}\Pi^{1/2}$ is diagonal with entries in $[c^{-1/2},c^{1/2}]$, the claimed multiplicative bounds on $\hat r_i$ follow immediately. Squaring the scores yields the $c^{-1}$ factor for the energy decrease.
\end{proof}

Thus a fixed multiplicative proxy degrades the greedy decrease by at most a factor $c^{-1}$ that depends only on the quality of the proxy. 

\medskip 
Finally, we quantify the costs as the cumulative number of edges used,  divided by the number of edges in the graph: one full power iteration has a normalized cost of $1$, while an iteration where a single node $i$ is updated has a normalized cost of $d^+_i/|E|$. Therefore, it is natural to include some weight that takes into account the cost of updating a node, which is its out degree $d^+_i$, and  choose the node $i$ that maximizes the decrease in energy \eqref{eq:en_decrease} per cost. We call this heuristic \texttt{GSD-deg} and its local version \texttt{LocalGSD-deg}. 

Table~\ref{tab:heuristics} summarizes all the new heuristics described above. 

\begin{table}[h]
\centering
\begin{tabular}{@{}ll@{}}
\toprule
\textbf{Heuristic} & \textbf{Selection Rule} \\ \midrule
\texttt{GSD} & \( B_t = \{\, i \mid i = \argmax_{j} \; |r_{t,j}|/\sqrt{ x_{t,j}} \,\} \) \\[1ex]
\texttt{GSD-deg}     & \( B_t = \{\, i \mid i = \argmax_{j} \; |r_{t,j}|/\sqrt{d^+_jx_{t,j}} \,\} \) \\[1ex]
\texttt{LocalGSD}      & $B_t = \{\, i \mid i = \argmax_{j \in N(i)} \; |r_{t,j}|/\sqrt{ x_{t,j}}$ \\[1ex]
\texttt{LocalGSD-deg}      & \( B_t = \{\, i \mid i = \argmax_{j \in N(i)} \; |r_{t,j}|/\sqrt{d^+_j x_{t,j}} \,\} \) \\[1ex]
\bottomrule
\end{tabular}
\caption{Here $r_t = x_t(I-P)$, $x_t$ is the current estimate of the stationary distribution, $N(i)$ denotes the neighborhood of node $i$, including itself, $d_i^+$ its the out-degree.}
\label{tab:heuristics}
\end{table}

We compare our new heuristics to the baseline heuristics summarized in Table~\ref{tab:baseline-heuristics}. The \texttt{Theta} heuristic is the best performing heuristic in \cite{RLGL}. 

\begin{table}[H]
\centering
\begin{tabular}{@{}lp{11cm}@{}}
\toprule
\textbf{Baseline heuristic} & \textbf{Selection Rule} \\ \midrule

\texttt{Gauss-Southwell} &
\( B_t = \{\, i \mid i = \arg\max_j |r_{t,j}| \,\} \) \\[1ex]

\texttt{PI} &
\( B_t = [n] \) \\[1ex]

\texttt{PCash} &
$B_t = \{\, i\} \text{ with probability } \propto |r_{t,i}|$\\[1ex]

\texttt{Theta} &
$ 
B_t = \{\, i \mid i \equiv t \pmod{n} \;\text{and}\ |r_{t,i}| \ge \theta_t(r) \,\},$
\\ 
& 
 $\hspace{1.5cm} \theta_t(r) = \left( \frac{ \sum_j |r_{t_0,j}|^r }{n} \right)^{1/r},
\quad r \ge 1,\quad
t_0 = \left\lfloor \frac{t}{n} \right\rfloor
$ \\

\bottomrule
\end{tabular}
\caption{Baseline heuristics.}
\label{tab:baseline-heuristics}
\end{table}

\subsection{Numerical results}

In this section, we compare the existing and the newly proposed RLGL heuristics on web graphs, and two synthetic graphs: a stochastic block model and a scale free preferential attachment model. More details on the graphs used can be found in Table~\ref{tab:graphs}.

For each graph, we compute both the stationary distribution on their largest strongly connected component (lscc) and the PageRank \cite{brin1998anatomy}, which is the stationary distribution of $P_\varepsilon$ in \eqref{eq:mixture}, where $Q$ is a simple random walk on a (directed) graph, $R=\tfrac{1}{n}\mathbf{1}^\top\mathbf{1}$ and the damping factor is $\varepsilon=0.85$. We report the $\ell_1$ norm of the residual as a function of the \emph{normalized cost}.

\medskip

We start with comparison of the Gauss--Southwell rule with different rescalings of the residuals, on the lscc of the Harvard500 graph. As we  observe in Figure~\ref{fig:GS_rescaling}, rescaling by $\sqrt{x_{t,i}}$, which is our proposed GSD heuristic, results in the fastest convergence.

Across all the graphs we tested, our new heuristics, particularly \texttt{GSD-deg}, outperform the previously best-known methods, including the Theta heuristic of \cite{RLGL}. It is noteworthy that the local heuristics also perform consistently well, almost always surpassing the prior state of the art, despite using only local information.

\begin{figure}[H]
    \centering
    \includegraphics[width=0.5\linewidth]{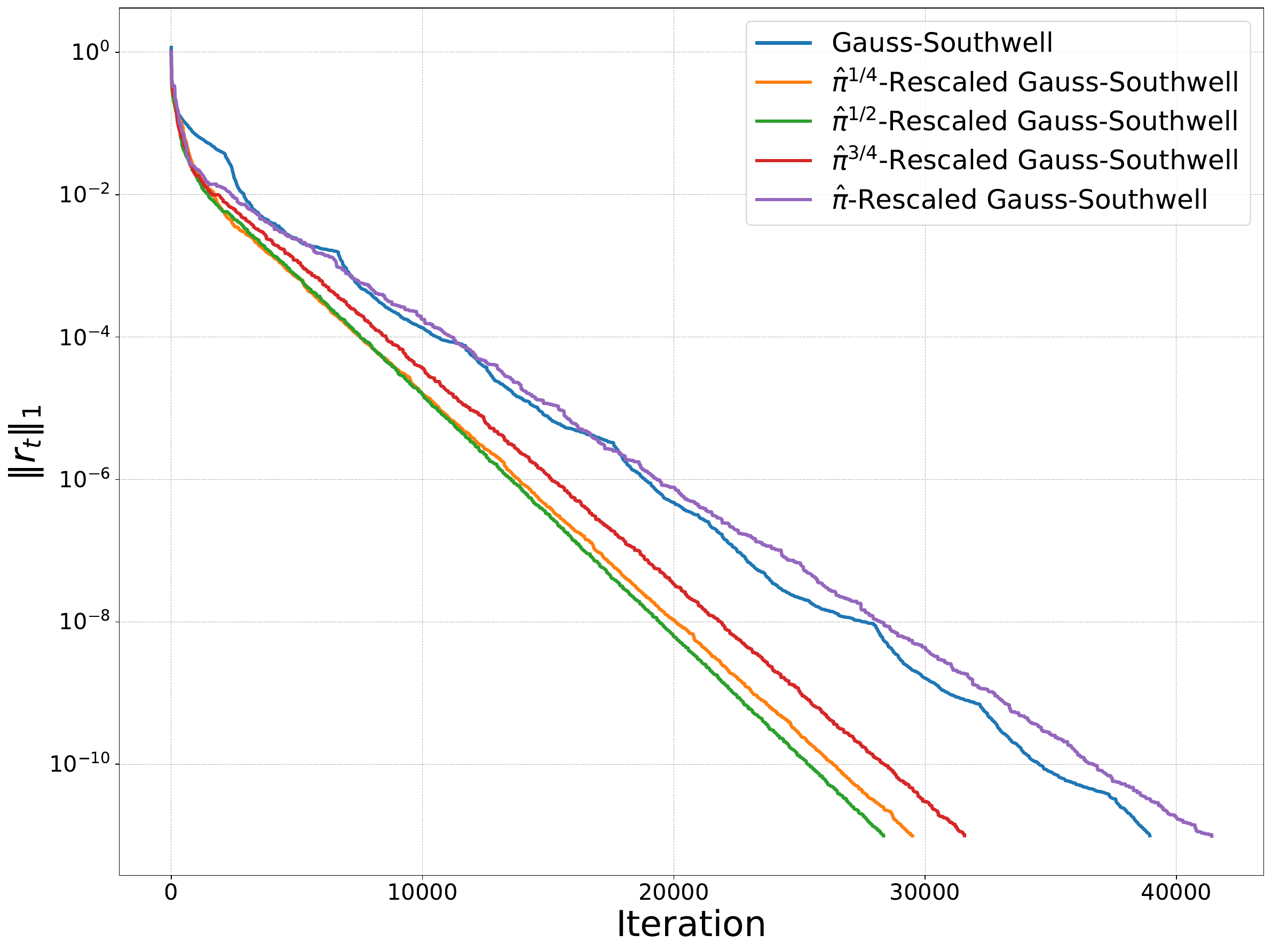}
    \caption{Stationary distribution computation of the Harvard500 lscc. Comparison of different residual rescalings for the Gauss-Southwell heuristic.  On the $y$-axis the $\ell_1$ norm of the residual, on the $x$-axis the iteration number. Rescaling by $\sqrt{\hat\pi}$ results in faster convergence.}
    \label{fig:GS_rescaling}
\end{figure}

\begin{figure}[H]
\centering
\begin{minipage}{0.48\linewidth}
    \centering
    \includegraphics[width=\linewidth]{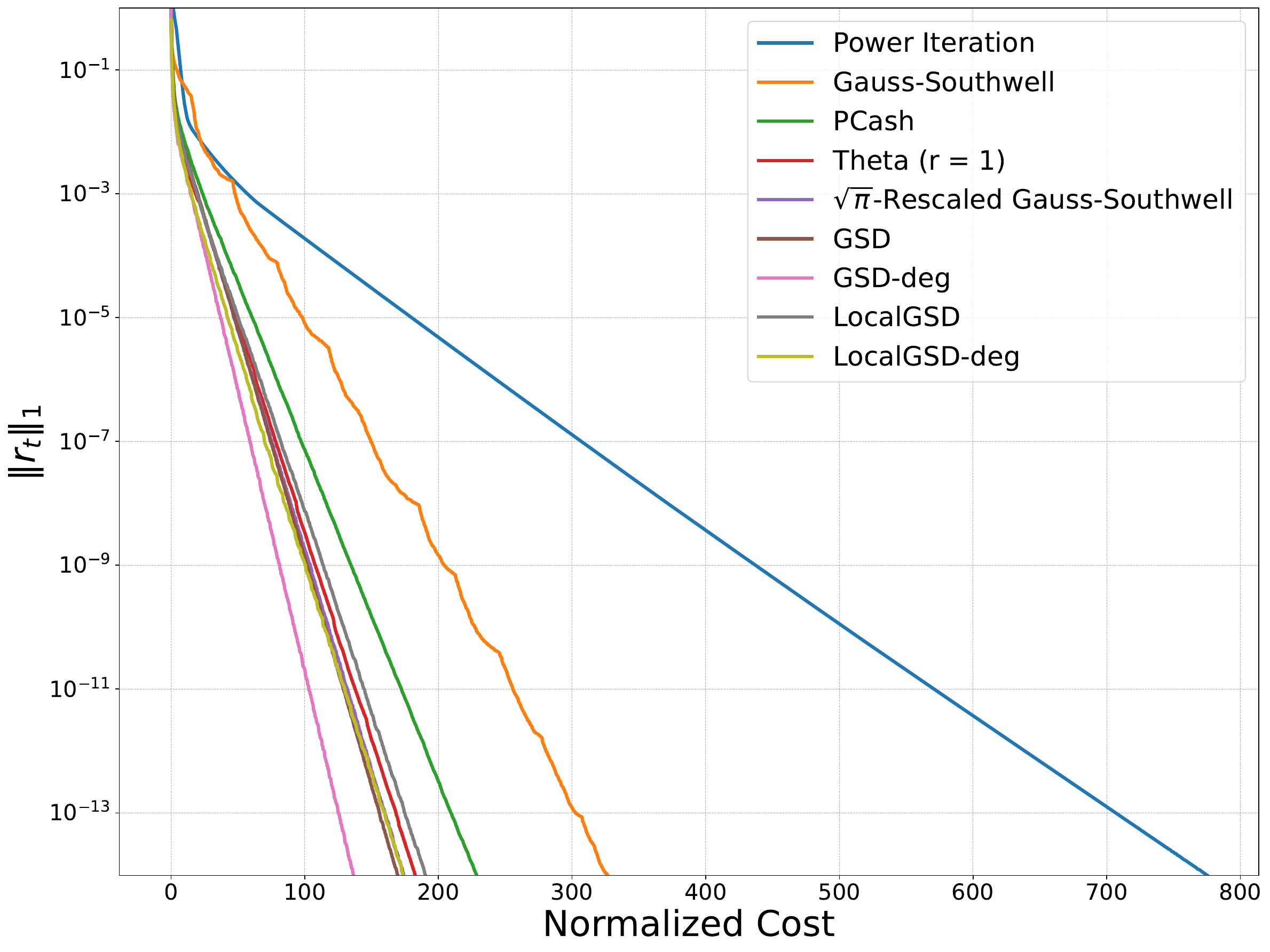}
    \caption{Stationary distribution on Harvard500 lscc.}
    \label{fig:harvard500_stationary}
\end{minipage}\hfill
\begin{minipage}{0.48\linewidth}
    \centering
    \includegraphics[width=\linewidth]{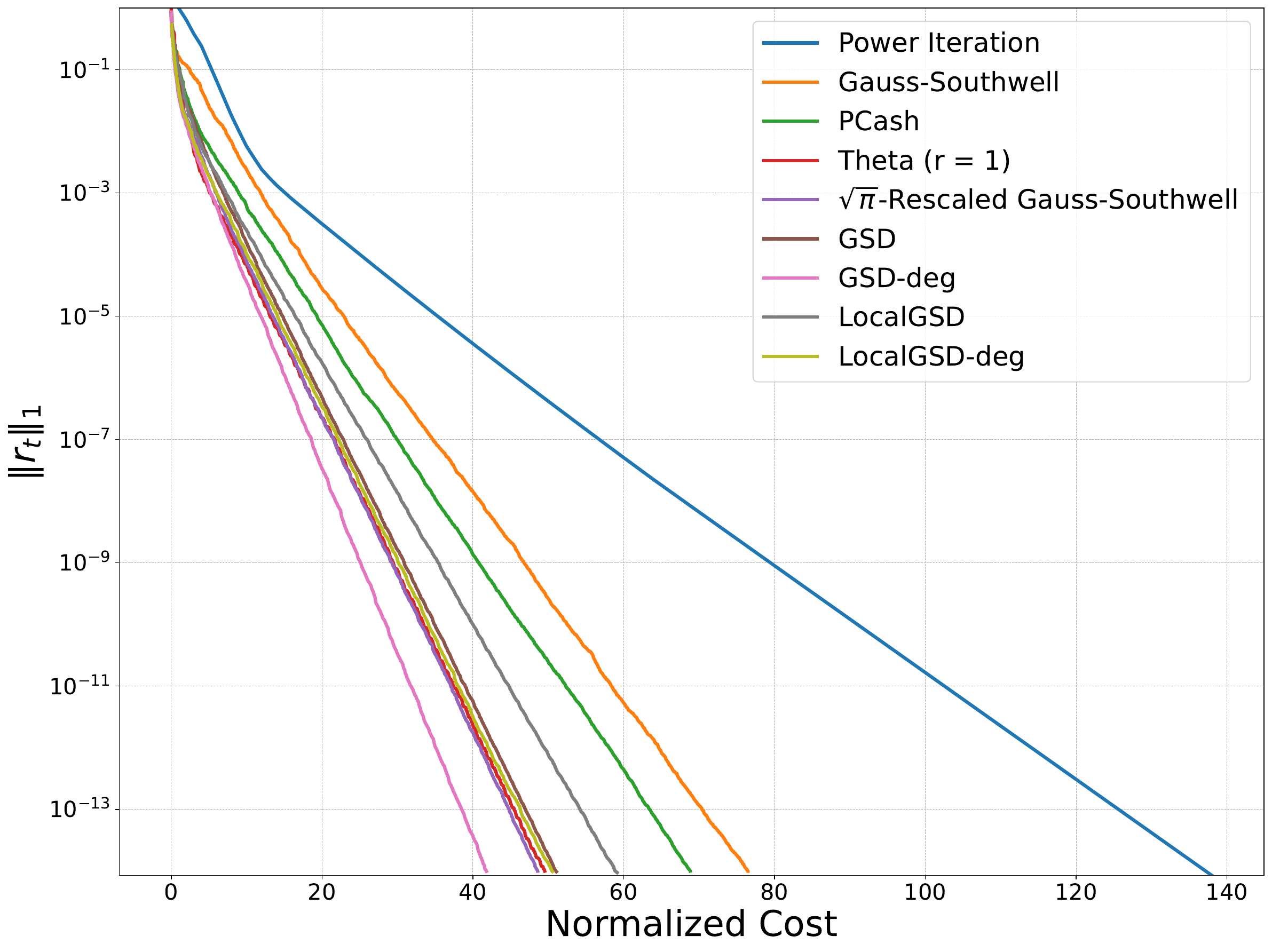}
    \caption{PageRank on Harvard500 lscc.}
    \label{fig:harvard500_pagerank}
\end{minipage}
\end{figure}

\begin{figure}[H]
\centering
\begin{minipage}{0.48\linewidth}
    \centering
    \includegraphics[width=\linewidth]{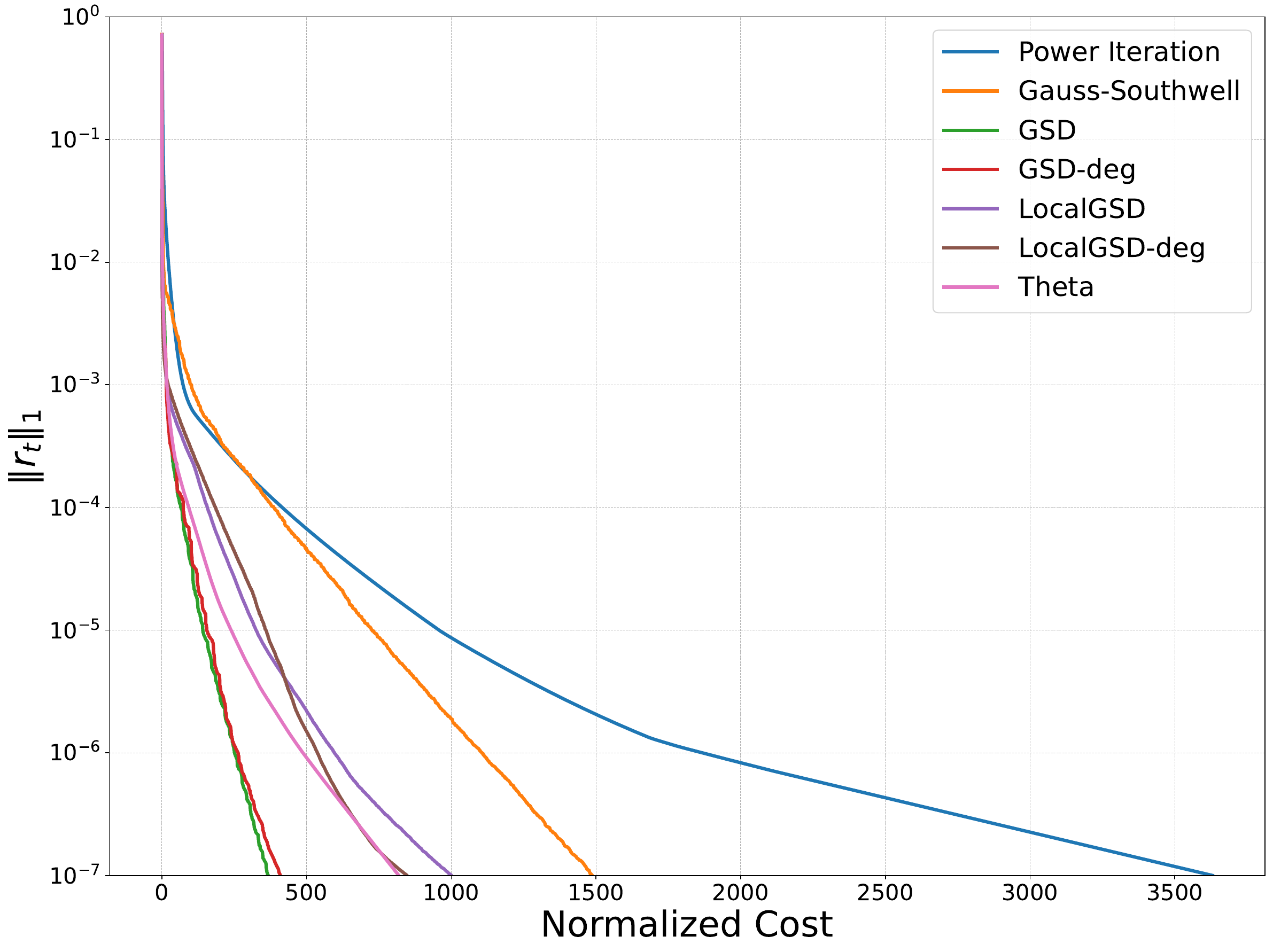}
    \caption{Stationary distribution computation on web-edu.}
    \label{fig:webedu_stationary}
\end{minipage}\hfill
\begin{minipage}{0.48\linewidth}
    \centering
    \includegraphics[width=\linewidth]{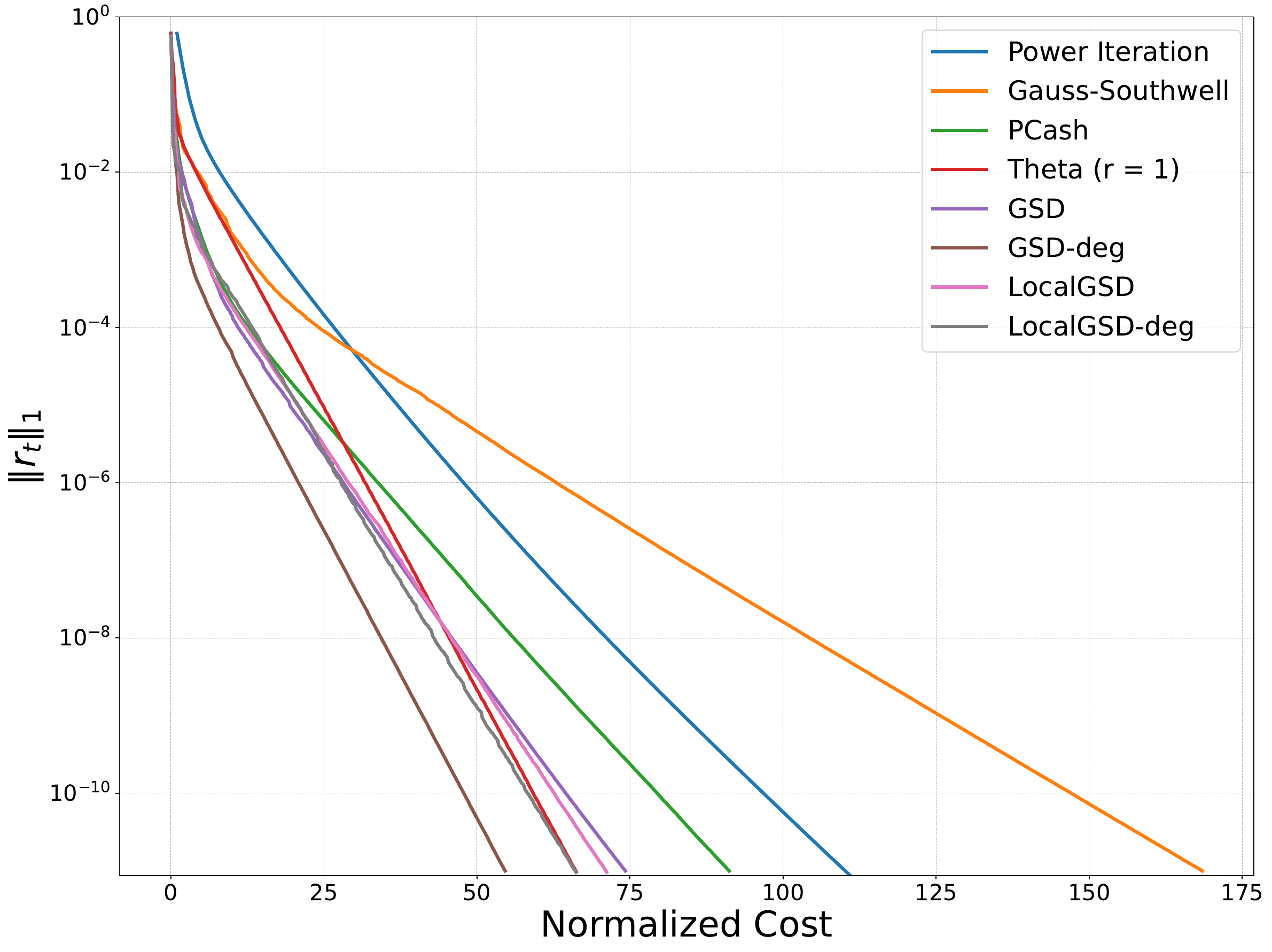}
    \caption{PageRank computation on web-edu.}
    \label{fig:webedu_pagerank}
\end{minipage}
\end{figure}

\begin{figure}[H]
\centering
\begin{minipage}{0.48\linewidth}
    \centering
    \includegraphics[width=\linewidth]{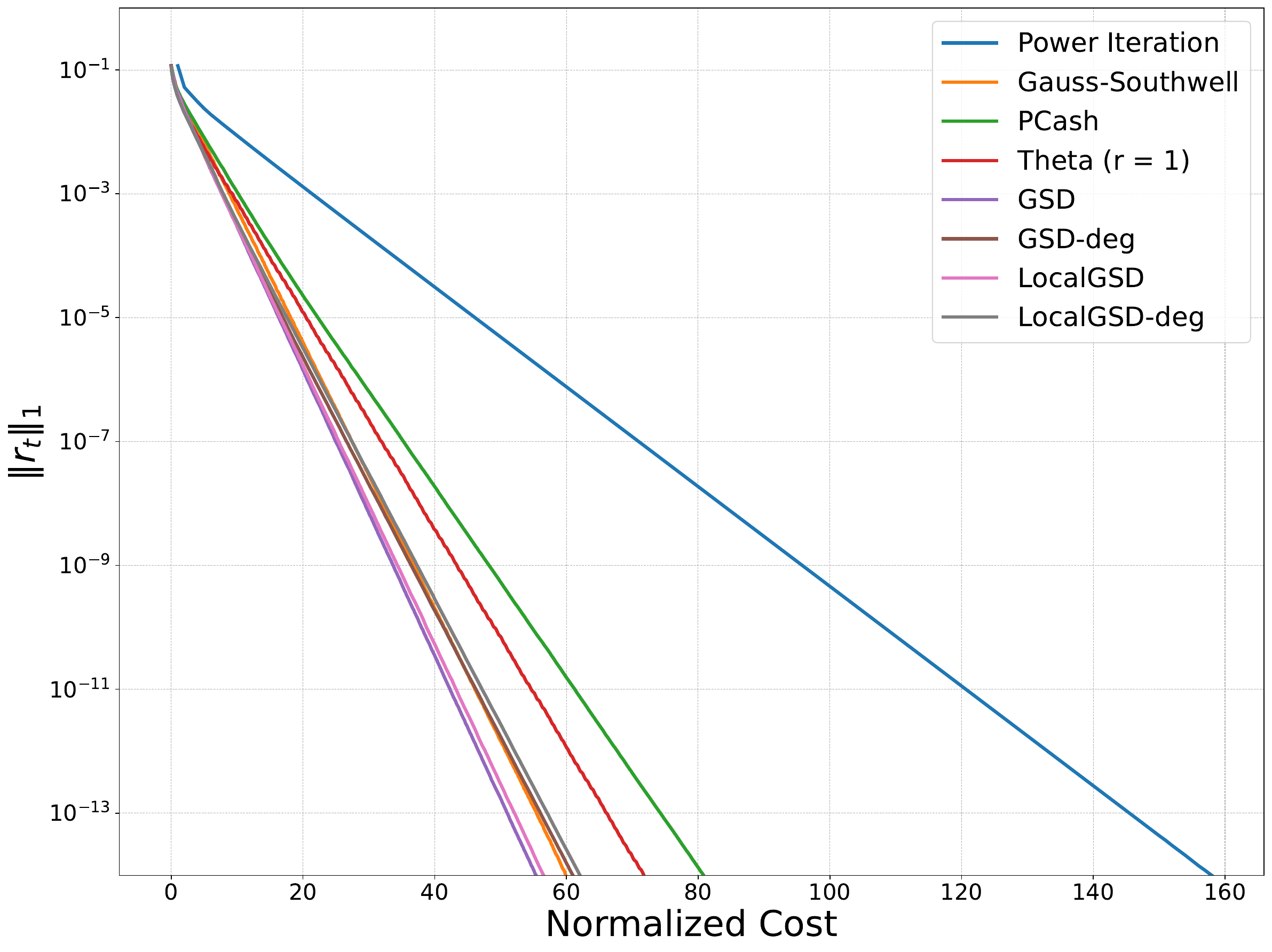}
    \caption{Stationary distribution computation SBM.}
    \label{fig:SBM_stationary}
\end{minipage}\hfill
\begin{minipage}{0.48\linewidth}
    \centering
    \includegraphics[width=\linewidth]{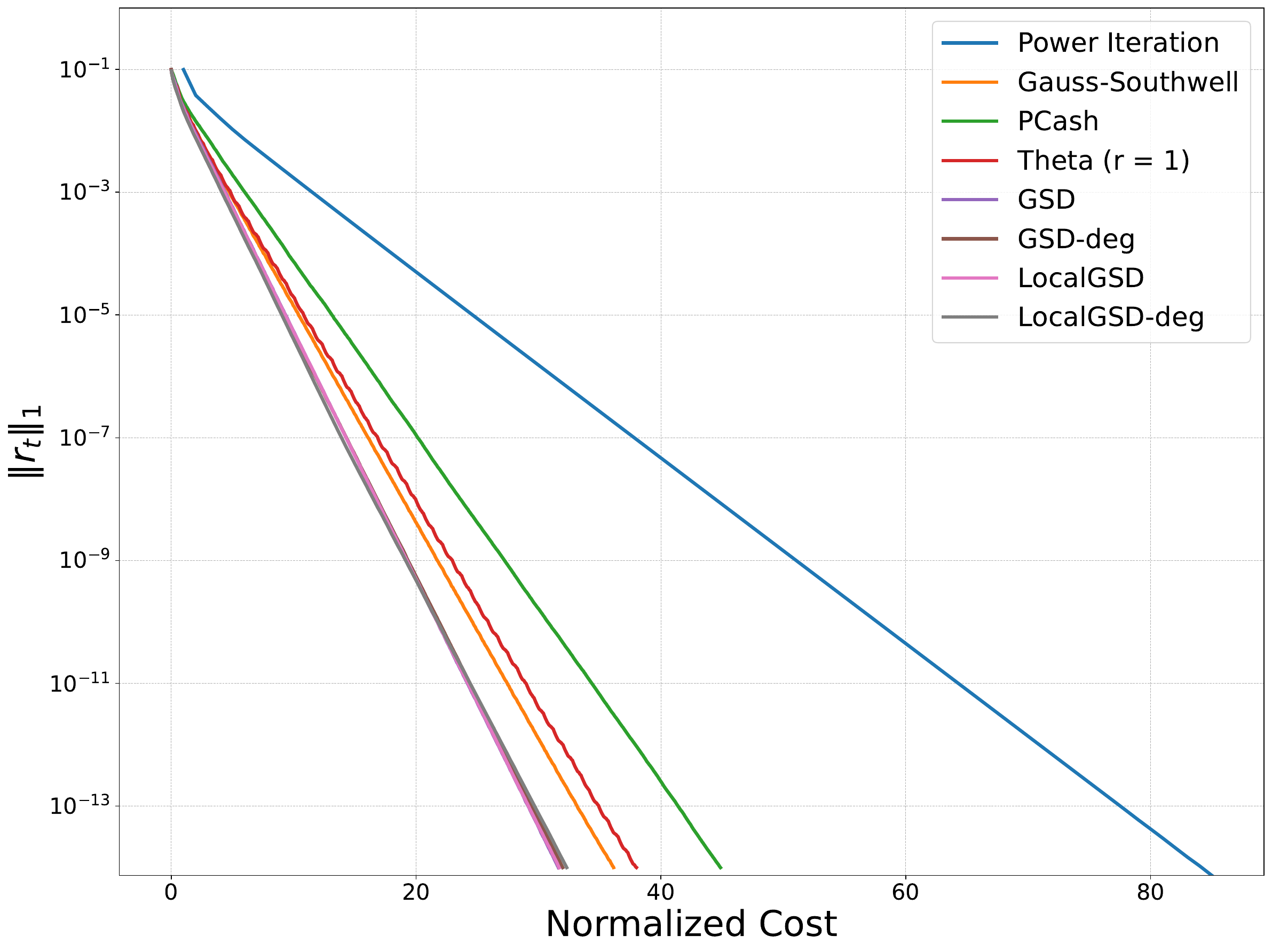}
    \caption{PageRank computation on SBM.}
    \label{fig:SBM_pagerank}
\end{minipage}
\end{figure}

\begin{figure}[H]
\centering
\begin{minipage}{0.48\linewidth}
    \centering
    \includegraphics[width=\linewidth]{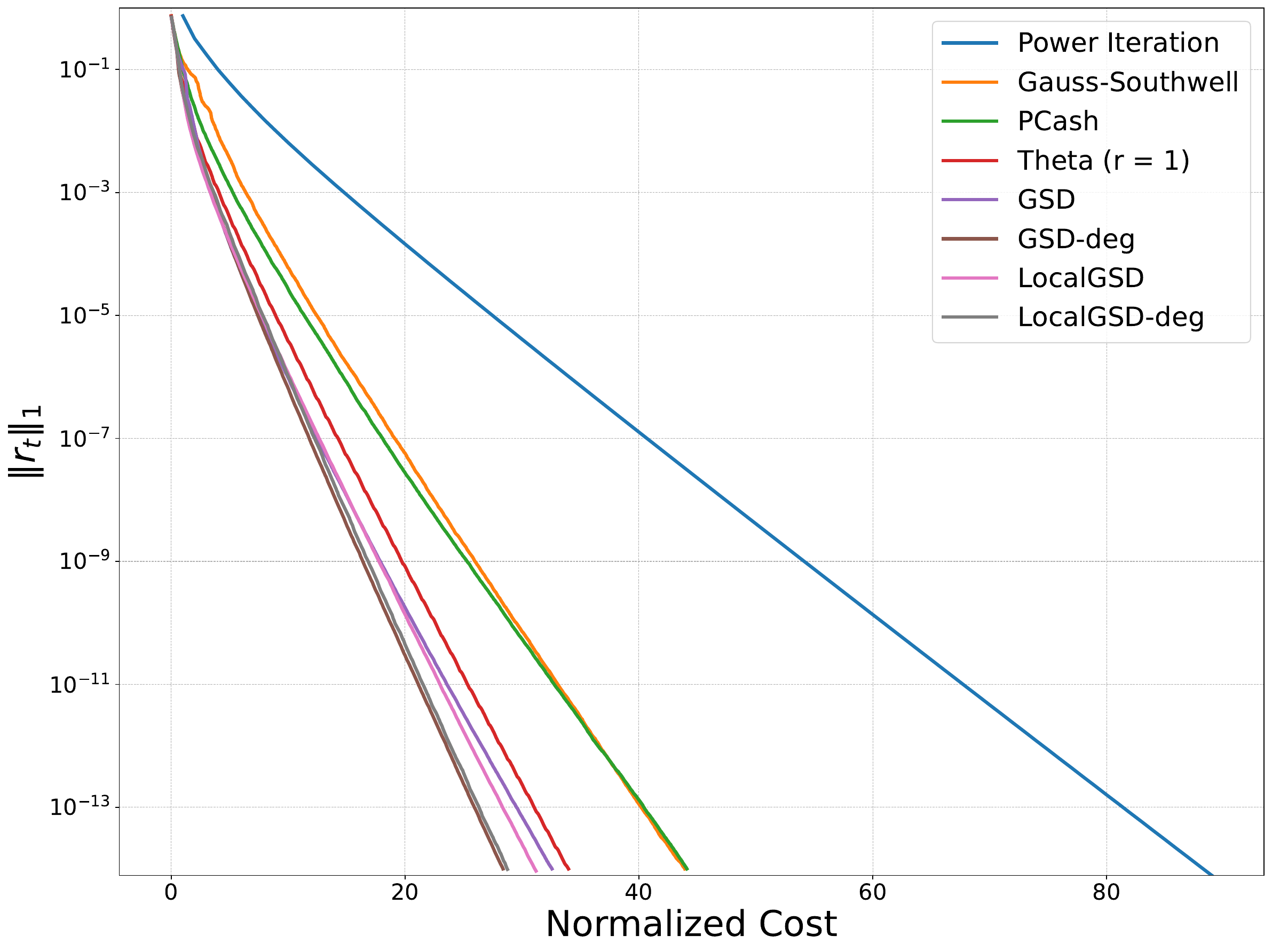}
    \caption{Stationary distribution computation on scale free.}
    \label{fig:scalefree_stationary}
\end{minipage}\hfill
\begin{minipage}{0.48\linewidth}
    \centering
    \includegraphics[width=\linewidth]{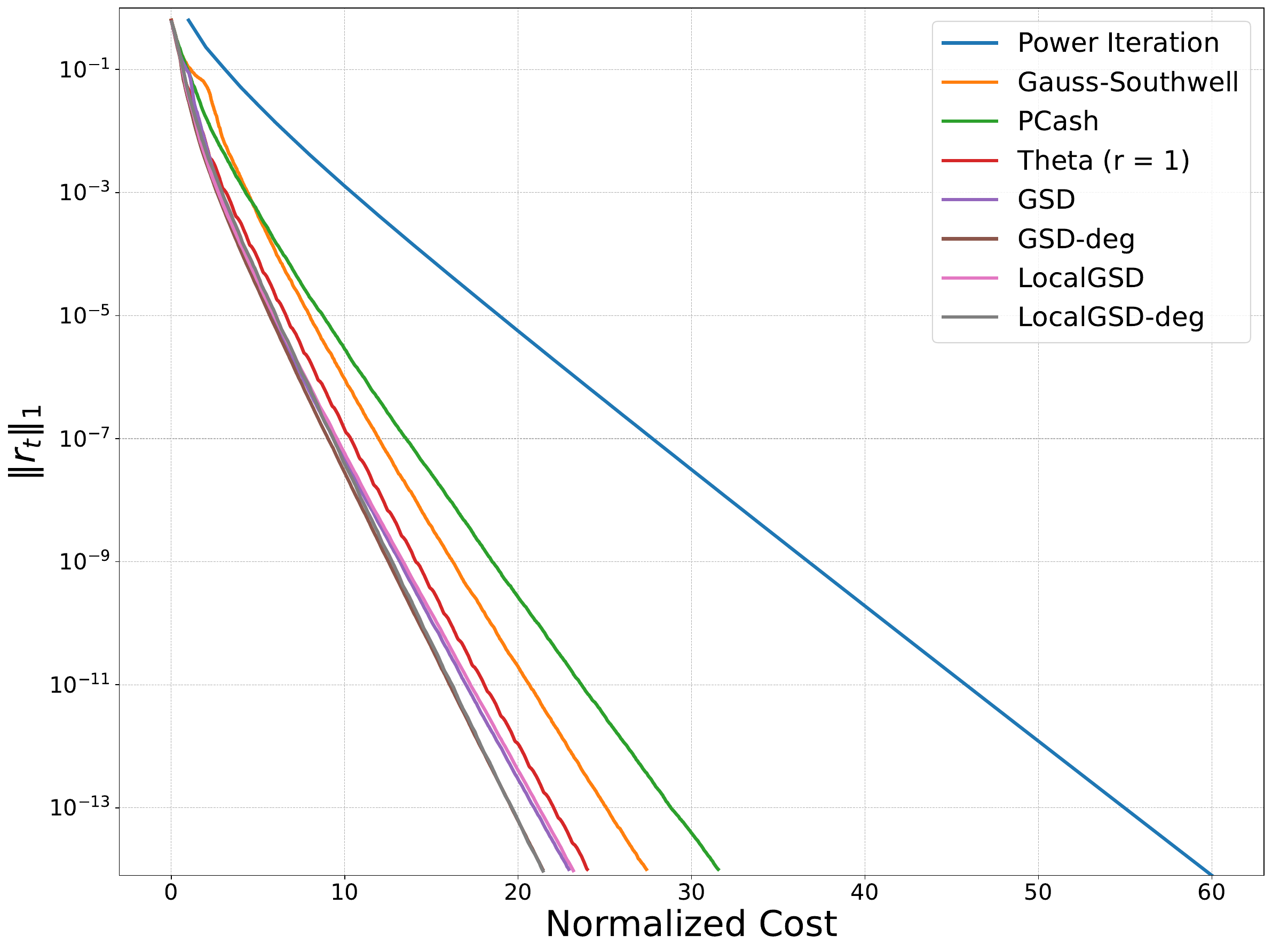}
    \caption{PageRank on computation scale free.}
    \label{fig:scalefree_pagerank}
\end{minipage}
\end{figure}

\begin{figure}[H]
    \centering
    \includegraphics[width=0.5\linewidth]{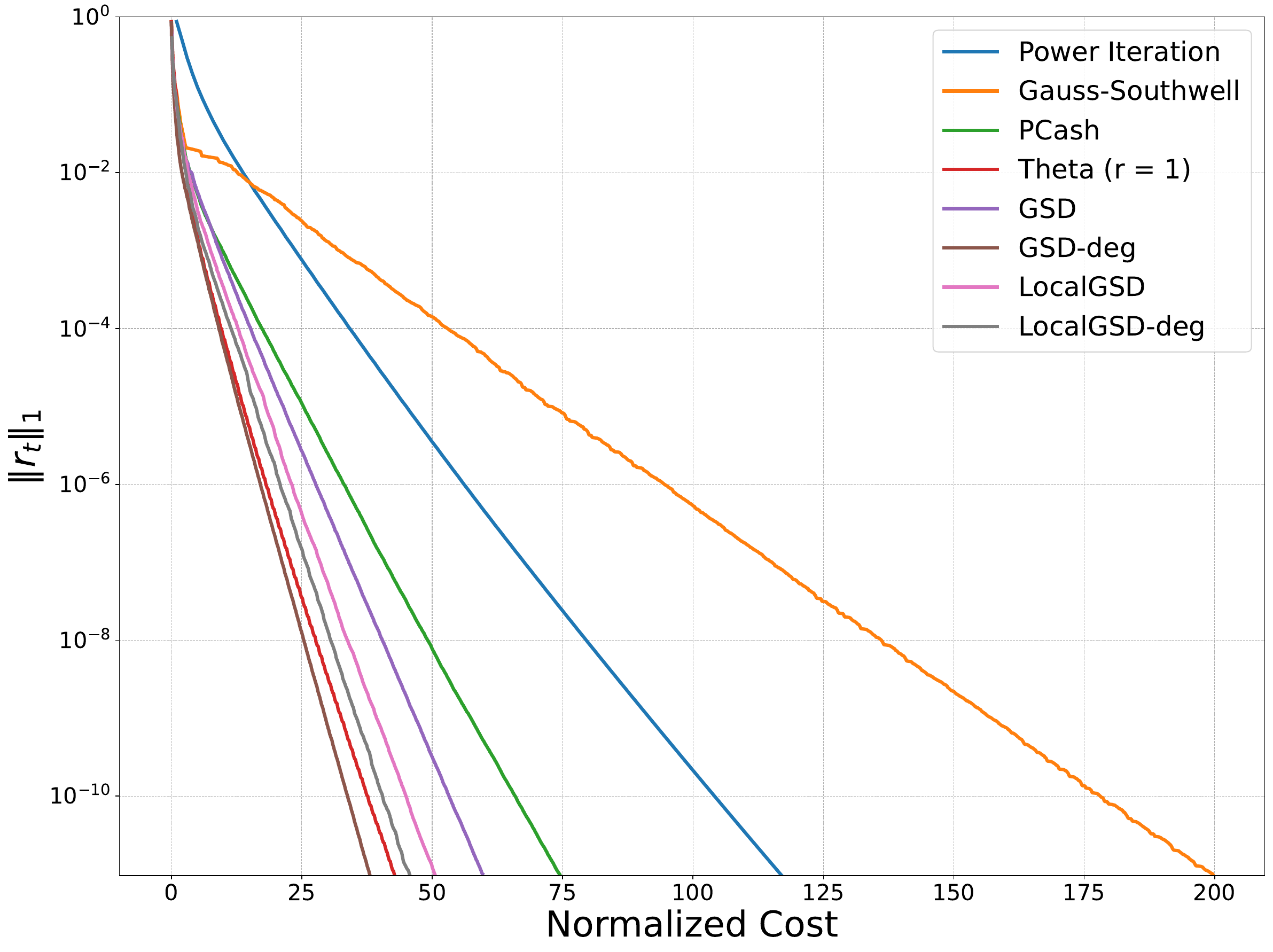}
    \caption{Page rank computation on wb-cs-stanford lscc.}
    \label{fig:wb-cs-stanford_pagerank}
\end{figure}

\begin{table}[h!]
\centering
\footnotesize
\begin{tabular}{lrrc}
\hline
\textbf{Graph} & \textbf{Nodes} & \textbf{Edges} & \textbf{Description} \\
\hline

Harvard500 lscc$^\dagger$          & 335    & 1963     & web graph  \\
SBM $^\ddagger$ & 800    & 27852    & synthetic block model \\
scale free $^\ddagger$               & 1000   & 2991     & synthetic scale-free network \\
wb-cs-stanford lscc$^*$      & 2759   & 13895    & web graph  \\
web-edu$^\dagger$                   & 3031   & 6474     & web graph \\

\hline
\end{tabular}
\caption{Graphs used in the numerical experiments. $^\dagger$: dataset from the Network Repository \cite{network_rep}; $^*$: dataset from SNAP \cite{snap}; $^\ddagger$: synthetic graphs generated by the authors.}
\label{tab:graphs}
\end{table}

\section{Conclusion and Further Research}

In this work, we have found the connection between the RLGL algorithm and coordinate descent method. Specifically, our Theorem~\ref{thm:exp_conv} says that when $P$ is similar to a symmetric matrix, then the RLGL update is a gradient update when minimizing a specific energy function. When $P$ is reversible, we have shown that this energy function is the Dirichlet energy of the Markov chain. Moreover, if updated coordinates form an independent set with respect to $P$, then the RLGL update matches the optimal step size of the coordinate descent. 

In light of computing stationary distributions of large Markov chains, which is the original goal of the RLGL algorithm, reversibility  is a restrictive assumption because the stationary distribution of a reversible chain can be easily computed by other methods. To address this, in Theorem~\ref{thm:perturbed_global}, we have extended our results to a class of so-called `nearly reversible' Markov chains,  using a perturbation approach. 

Even though the nearly reversibility condition is still rather stringent, in practice, the heuristics motivated by the Dirichlet energy minimization, improve the performance of RLGL method. In our numerical experiments on both synthetic and real-world networks, these new heuristics consistently outperform the baselines. In particular, \texttt{GSD-deg} always surpasses the previous state of the art. Moreover, the local variant, \texttt{LocalGSD-deg}, almost always outperforms the prior SOTA, despite relying only on local information.

 A natural next step is to look for conditions on $P$ that are weaker than near reversibility, and yet  allow for an energy–minimization interpretation of the RLGL updates, or at least guarantee some form of descent. Another question is whether one can identify structural properties of directed chains for which coordinate-based updates admit provable convergence guarantees.
\printbibliography

\newpage
\section*{Appendix}
\label{sec:appendix}

\begin{defn} A differentiable function $f : \R\to \R$ is said to be $\mu$-\emph{strongly convex} if there exists $\mu > 0$, such that
\begin{equation}
    f(y) \geq f(x) + \nabla f(y)^\top (y-x) + \frac{\mu}{2}\Vert y-x\Vert_2^2,\qquad \forall x,y \in \R.
    \label{eq:convexity}
\end{equation}
\end{defn}

\begin{lemma}
\label{lem:stronly-convex}
    Let $f$ be a differentiable $\mu$-strongly convex function, and let $f^* = \inf_{x\in \R^n}$. Then,
    $$
    \frac{1}{2\mu} \Vert \nabla f(x)\Vert_2^2 \geq f(x) - f^*
    $$
\end{lemma}
\begin{proof}
    Minimize both sides of the inequality \eqref{eq:convexity}.
\end{proof}

\begin{lemma}
\label{lem:PSD-PL}
Let $Q \in \R^{n\times n}$ be symmetric and positive semi-definite. Then the quadratic form
$$
f(x)= \frac12xQx^\top
$$
satisfies the PL inequality with $\mu = \min\{\lambda > 0 \mid \lambda \text{ eigenvalue of } Q\}$. 
\end{lemma}
 
\begin{proof}
    Since the gradient is $\nabla f(x) = xQ$, we see that 
    $$
    \frac12\Vert \nabla f(x)\Vert^2 = \frac12 xQ^2x^\top,
    $$
    thus the PL-inequality is equivalent to 
    $$
    xQ^2x^\top \geq \mu\, x Q x^\top.
    $$
    We can show this diagonalizing $Q = U^\top \Lambda U$ with $\Lambda = \diag(\lambda_1,\dots,\lambda_n)$, and let $\mu$ be the smallest non-zero eigenvalue. Let $y = xU^\top$. Then
    $$
    xQ^2x^\top = \sum_{i=1}^n \lambda_i^2 y_i^2 \geq \mu \sum_{i=1}^n \lambda_i y_i^2 = \mu xQx^\top.
    $$
\end{proof}

We will also need Lipschitzianity of the gradient. We define the lipschitz constants for each coordinate
$$
\vert \partial_if(x + \lambda e_i) - \partial_i f(x) \vert \leq L_i |\lambda |, \qquad i=1,\dots,n,
$$

and define $\qquad L_{max} := \max_i L_i$.

Lipschitzianity of the gradient allows us to bound the growth of a function, the next result is known as the \emph{descent lemma}.

\begin{lemma}
    Let $f: \R^n\to \R$ be differentiable and with gradient $L_1,L_2,\dots,L_n$ lipschitz constant. Then
    $$
    f(x+\lambda e_i) \leq f(x) + \lambda\partial_i f(x) + \frac{L_i}{2}\lambda^2
    $$
\end{lemma}

\begin{proof}
    Define the function $g: \R \to \R$, $g(t) = \partial_i f(x+te_i)$. Then,
    \begin{align*}
    g(\lambda) - g(0) &= \int_0^\lambda g'(t)dt \\ 
                    &\leq \int_0^\lambda |g'(t)-g'(0)|dt +\int_0^\lambda g'(0)dt \\ 
                    &\leq L_i\int_0^\lambda |t| dt + \lambda g'(0)\\
                    &\leq \frac{L_i}{2} \lambda^2 + \lambda g'(0)\\
    \end{align*}
\end{proof}

\begin{lemma}
Let $A,B \in \mathbb{R}^{n \times n}$ be symmetric positive semidefinite matrices such that $\ker(A) = \ker(B).$ Then there exist constants $0 < c_1 \leq c_2 < \infty$ such that for all $x \in \mathbb{R}^n$,
$$
c_1\, x B x^\top \;\leq\; x A x^\top \;\leq\; c_2\, x B x^\top.
$$
\label{lemma:quad_bound}
\end{lemma}

\begin{proof}
Let $K := \ker(A) = \ker(B)$ and $V = K^\perp$. For $x \in \mathbb{R}^n$ write $x = u + k$ with $u \in V$, $k \in K$. Since $A$ and $B$ vanish on $K$, we have
\[
x^\top A x = u^\top A u, \qquad x^\top B x = u^\top B u.
\]
On $V$, both $A$ and $B$ are positive definite, so the generalized Rayleigh quotient
\[
\phi(u) := \frac{u^\top A u}{u^\top B u}, \quad u \in V \setminus \{0\},
\]
is strictly positive. Restricting $\phi$ to the compact unit sphere in $V$, $S := \{ u \in V : \|u\|=1 \}$, the continuous function $\phi$ attains a minimum $c_1 > 0$ and a maximum $c_2 < \infty$. Hence,
$$
c_1 \, u^\top B u \le u^\top A u \le c_2 \, u^\top B u, \quad \forall u \in V,
$$
and therefore the same inequality holds for all $x \in \mathbb{R}^n$. 
\end{proof}

\begin{proof}[Proof of Corollary~\ref{cor:random_greedy_explicit}] 
\label{app:random_greedy_proof}
We seek conditions under which Theorem \ref{thm:rlgl_conv} guarantees convergence. This requires two things:
\begin{enumerate}
\item We must effectively transfer the goodness from the observed residual to the symmetric gradient (Proposition \ref{prop:beta_transfer}).
\item The perturbation must be small enough to satisfy the convergence condition (Theorem \ref{thm:perturbed_global}).
\end{enumerate}

Theorem \ref{thm:rlgl_conv} guarantees convergence if:
$$
\kappa < \mu\sqrt{\beta_{\text{sym}}}.
$$
Using the relation $\kappa \le \mu\sqrt{n}\eta_\infty$, this condition becomes:
\begin{equation}
\label{eq:proof_cond_1}
\sqrt{n}\eta_\infty < \sqrt{\beta_{\text{sym}}}.
\end{equation}

We use Proposition \ref{prop:beta_transfer} to substitute $\sqrt{\beta_{\text{sym}}}$. Let $B = \sqrt{\beta_{{\rm norm}}} = 1/\sqrt{n}$. The proposition states:
$$
\sqrt{\beta_{\text{sym}}} = \frac{B - \gamma^\star_\infty}{1 + \gamma^\star_\infty}, \quad \text{where} \quad \gamma^\star_\infty = \frac{\sqrt{n}\eta_\infty}{1-\sqrt{n}\eta_\infty}.
$$
To simplify notation, let $X = \sqrt{n}\eta_\infty$. Then $\gamma^\star_\infty = \frac{X}{1-X}$.
Substituting this into the expression for $\sqrt{\beta_{\text{sym}}}$:
$$
\sqrt{\beta_{\text{sym}}} = B(1-X) - X.
$$
We now substitute this simplified form back into our convergence condition \eqref{eq:proof_cond_1}:
$$
2X < B(1-X) \implies 2X < B - BX \implies X(2+B) < B.
$$
Thus, we require:
$$
X < \frac{B}{2+B}.
$$
Substituting back $X = \sqrt{n}\eta_\infty$ and $B = 1/\sqrt{n}$ gives the bound:
$$
\eta_\infty < \frac{1}{n(2 + 1/\sqrt{n})} = \frac{1}{2n + \sqrt{n}}.
$$
\end{proof}

\end{document}